\documentclass[11pt]{amsart}
\pdfoutput=1

\usepackage{amssymb}
\usepackage{microtype}
\usepackage{mathtools}
\usepackage[margin=1.25in]{geometry}

\usepackage[numbers]{natbib}
\usepackage{enumitem}
\usepackage{aliascnt}

\usepackage{grffile}

\usepackage{hyperref}
\hypersetup{
  pdftitle={Structure-preserving numerical integrators for Hodgkin--Huxley-type systems},
  pdfauthor={Zhengdao Chen, Baranidharan Raman, and Ari Stern},
  pdfsubject={MSC 2010: 65P30, 37M20, 92C20},
  bookmarksopen=false,
}

\theoremstyle{plain}

\theoremstyle{proposition}
\newaliascnt{proposition}{theorem}
\newtheorem{proposition}[proposition]{Proposition}
\aliascntresetthe{proposition}
\theoremstyle{corollary}
\newaliascnt{corollary}{theorem}
\newtheorem{corollary}[corollary]{Corollary}
\aliascntresetthe{corollary}

\theoremstyle{definition}
\newaliascnt{definition}{theorem}
\newtheorem{definition}[definition]{Definition}
\aliascntresetthe{definition}
\newaliascnt{example}{theorem}
\newtheorem{example}[example]{Example}
\aliascntresetthe{example}

\theoremstyle{remark}
\newaliascnt{remark}{theorem}
\newtheorem{remark}[remark]{Remark}
\aliascntresetthe{remark}
\theoremstyle{lemma}
\newaliascnt{lemma}{theorem}

\aliascntresetthe{lemma}

\begin{document}
\author{Zhengdao Chen}
\address{Department of Mathematics,
  Courant Institute of Mathematical Sciences, New York University}
\email{zc1216@nyu.edu}

\author{Baranidharan Raman}
\address{Department of Biomedical Engineering, Washington University in St.~Louis}
\email{barani@wustl.edu}

\author{Ari Stern}
\address{Department of Mathematics and Statistics, Washington University in St.~Louis}
\email{stern@wustl.edu}

\title[Structure-preserving integrators for Hodgkin--Huxley-type systems]{Structure-preserving numerical integrators for Hodgkin--Huxley-type systems}


\begin{abstract}
  Motivated by the Hodgkin--Huxley model of neuronal dynamics, we
  study explicit numerical integrators for ``conditionally linear''
  systems of ordinary differential equations. We show that splitting
  and composition methods, when applied to the Van~der~Pol oscillator
  and to the Hodgkin--Huxley model, do a better job of preserving
  limit cycles of these systems for large time steps, compared with
  the ``Euler-type'' methods (including Euler's method, exponential
  Euler, and semi-implicit Euler) commonly used in computational
  neuroscience, with no increase in computational cost. These limit
  cycles are important to preserve, due to their role in neuronal
  spiking. Splitting methods even compare favorably to the explicit
  exponential midpoint method, which is twice as expensive per step.
  The second-order Strang splitting method is seen to perform
  especially well across a range of non-stiff and stiff dynamics.
\end{abstract}

\maketitle

\section{Introduction}
\label{sec:intro}

This paper is concerned with the numerical integration of
\emph{conditionally linear} systems of ordinary differential equations
(ODEs) in $ \mathbb{R}^{d} $, which are systems of the form
\begin{equation}
  \label{eqn:cl}
  \dot{ x } _i = a _i ( x ) x _i + b _i (x) , \qquad i = 1 , \ldots, {d} ,
\end{equation}
where $ a _i , b _i $ are real-valued functions depending only on
$ x _j $ for $ j \neq i $. These systems have the defining property
that, if all $ x _j $ with $ j \neq i $ are stationary, then $ x _i $
satisfies a first-order linear ODE with constant
coefficients.\footnote{It is straightforward to generalize what
  follows to the case where $ x _i $ and $ b _i $ are vector-valued
  and $ a _i $ is matrix-valued, so that $ x _i $ satisfies a
  first-order linear \emph{system} of ODEs with constant coefficients
  when $ x _j $ is stationary for $ j \neq i $, but the scalar case
  covers the applications we are interested in.}

Our motivation comes from neuronal dynamics, where ``conditional
linearity is a fairly generic property of nonlinear neuronal models''
(\citet{MaSh1998}). Since these models are nonlinear, computational
neuroscientists must use numerical simulation to study the dynamical
behavior of individual model neurons, as well as networks of such
neurons. We will focus in particular on the model of \citet{HoHu1952},
which is one of the most well-known and widely-used biological neuron
models.

To simulate \eqref{eqn:cl} as efficiently as possible, one wishes to
minimize the number of evaluations of the nonlinear functions $ a _i $
and $ b _i $. (This becomes especially important when $ {d} $ is
large, as is the case for large biological neural networks.) It is
therefore desirable to use an explicit numerical integrator that
allows for large time step sizes while still producing sufficiently
accurate dynamics.  One of the main obstacles to doing so is
\emph{stiffness}: when $ a _i $ is a large negative number, a
traditional explicit Runge--Kutta method (like Euler's method) becomes
numerically unstable unless the time step size is very small. For this
reason, various authors have proposed using explicit methods designed
for stiff systems, such as exponential integrators and semi-implicit
integrators; \citet{BoNe2013} note that such integrators remain stable
for much larger time step sizes when simulating Hodgkin--Huxley
neurons, on the order of 1~ms, whereas traditional methods require
time step sizes on the order of 0.01~ms. They also present numerical
evidence that ``[a]daptive time-stepping is of questionable use'' when
simulating networks of neurons \citep[Section 7]{BoNe2013}, so the
ability to take large time steps is the major arbiter of numerical
efficiency.

However, there is another obstacle to taking large time step sizes,
having to do with preserving the qualitative dynamics of neuronal
spiking in the Hodgkin--Huxley model.  When the input current into a
neuron is low, the membrane voltage is attracted to a resting
equilibrium; however, when the input current exceeds a threshold, the
voltage begins rapidly rising and falling periodically. These voltage
spikes, called action potentials, are the mechanism by which neurons
send signals to one another. From a dynamical systems point of view,
this corresponds to a bifurcation: the ``resting'' fixed point becomes
unstable, and the system is attracted to a stable ``spiking'' limit
cycle lying on a two-dimensional center manifold
(\citet{Hassard1978,Izhikevich2007}).  In order to simulate these
dynamics faithfully and efficiently, it is therefore desirable that a
numerical integrator be able to preserve these limit cycles at large
time step sizes. Yet, Euler's method does a poor job of preserving
limit cycles in nonlinear dynamical systems, even for simple systems
like the Van~der~Pol oscillator, unless one takes very small time
steps---even smaller than one would need for numerical stability
(\citet{HaLu1999}).

In this paper, we investigate this problem of limit cycle
preservation. We show that the exponential Euler and semi-implicit
Euler (SI Euler) methods, while superior to the traditional Euler
method in terms of stability, have the same fundamental problem of
non-preservation of limit cycles. We introduce a new family of
splitting and composition methods specifically designed for
conditionally linear systems. These methods are just as efficient as
the exponential and SI Euler methods, with comparable stability
behavior, but they do much better at preserving limit cycles. The
splitting methods even compare favorably to the exponential midpoint
method proposed by \citet{BoNe2013}, at half the computational
cost. These properties are demonstrated both theoretically (using
backward error analysis) and numerically for the Van~der~Pol
oscillator, and numerically for the Hodgkin--Huxley model neuron. In
particular, a Strang splitting method emerges as the best performer
among all the methods considered.

\section{Numerical integrators for conditionally linear systems}

\subsection{Preliminaries}
\label{sec:prelim}

All of the numerical integrators that we consider, old and new, are
based on the fact that linear ODEs with constant coefficients can be
solved exactly in closed form. That is, 
\begin{equation*}
  \dot{x} = a x + b 
\end{equation*} 
has the time-$h$ flow given by the explicit formula
\begin{equation*}
  x ( t + h ) = \exp ( h a ) x (t) + \frac{ \exp ( h a ) - 1 }{ h a } h b .
\end{equation*}
Here, $ z \mapsto ( \exp z - 1 ) / z $ actually means the entire
function obtained by removing the singularity at $ z = 0 $, taking
$ 0 \mapsto 1 $. (This is sometimes called the ``relative error
exponential'' and is implemented in many numerical libraries as
\texttt{exprel}.)

In fact, we can also explicitly compute the flow of any Runge--Kutta
method applied to this ODE---even an implicit Runge--Kutta method---by
replacing the exponential function above with the stability function
$ z \mapsto r (z) $. For instance, Euler's method has stability
function $ r (z) = 1 + z $, so the time-$h$ flow map taking
$ x ^n \mapsto x ^{ n + 1 } $ is given by
\begin{equation*}
  x ^{ n + 1 } = ( 1 + h a ) x ^n + h b = x ^n + h ( a x ^n + b ) .
\end{equation*} 
More interesting is the backward Euler method, with
$ r (z) = 1/ ( 1 - z ) $; while an implicit method for nonlinear ODEs,
in this case it admits the explicit formula
\begin{equation*}
  x ^{ n + 1 }  = \frac{ 1 }{ 1 - h a } x ^n + \frac{ 1 }{ 1 - h a } h b ,
\end{equation*}
which rearranges to the more familiar expression,
\begin{equation*}
  x ^{ n + 1 } = x ^n + h ( a x ^{ n + 1 } + b ) .
\end{equation*} 
For splitting/composition methods, it can sometimes be desirable to
compute an approximate flow, even when the exact flow is
computable. For example, the IMEX method for highly-oscillatory
problems (\citet{StGr2009,McSt2014}) uses the midpoint method instead
of the exact flow of a harmonic oscillator, and the Boris method for
charged particle dynamics uses the midpoint method instead of the
exact flow of the magnetic field (\citet[Eq.~1.4]{HaLu2018}).

As mentioned in the introduction, the defining property of a
conditionally linear system \eqref{eqn:cl} is that, if all $ x _j $
with $ j \neq i $ are stationary, then $ x _i $ satisfies a
first-order linear ODE with constant coefficients. We now introduce
some notation to help formalize this. Write \eqref{eqn:cl} as
$ \dot{x} = f (x) $, where
$ f \colon \mathbb{R}^{d} \rightarrow \mathbb{R}^{d} $ is the vector
field with components
\begin{equation*}
  f _i (x) = a _i (x) x _i + b _i (x) , \qquad i = 1 , \ldots, {d} .
\end{equation*}
For $ i = 1, \ldots, {d} $, let
$ f ^{ ( i ) } \colon \mathbb{R}^{d} \rightarrow \mathbb{R}^{d} $ be
the vector field
\begin{equation*}
  f ^{ (i) } _j (x) =
  \begin{cases}
    f _i (x), & \text{if } i = j ,\\
    0 , &\text{if } i \neq j .
  \end{cases}
\end{equation*}
Let
$ \varphi ^{ (i) } _h \colon \mathbb{R}^{d} \rightarrow \mathbb{R}^{d}
$ denote the exact time-$h$ flow of $ f ^{ (i) } $. Since
$ f ^{ (i) } $ holds $ x _j $ stationary for $ j \neq i $, it follows
that this flow can be written in closed form as
\begin{equation*}
  \varphi _{h,j} ^{ (i) } (x) =
  \begin{cases}
    \exp \bigl( h a _i (x) \bigr) x _i + \dfrac{ \exp \bigl( h a _i (x)
      \bigr) - 1 }{ h a _i (x) } h b _i (x) , &\text{if } i = j ,\\
    x _j ,&\text{if } i \neq j .
  \end{cases}
\end{equation*}
(Here, we use $ \varphi ^{ (i) } _{h,j} $ to denote the $ j $th
component of $ \varphi ^{ (i) } _h $.) Following the discussion above,
we might instead approximate $ \varphi _h ^{ (i) } $ by an approximate
flow $ \Phi _h ^{ (i) } $, e.g., by applying a Runge--Kutta method to
$ f ^{ (i) } $ with time step size $h$, in which case we would simply
replace the exponential function in this formula by the stability
function of the method. Note that we may use a different Runge--Kutta
method (or the exact flow) for each $i = 1 , \ldots , {d}$.

We have therefore decomposed the vector field $f$ we wish to integrate
as
\begin{equation*}
  f = f ^{ (1) } + \cdots + f ^{ ({d}) } .
\end{equation*} 
Due to conditional linearity, we may integrate each $ f ^{ (i) } $
exactly---or with an arbitrary, possibly-implicit Runge--Kutta
method---with only a single evaluation of each of the nonlinear
functions $ a _i , b _i $. The motivating idea, for all of the methods
considered below, is to combine these flows in such a way as to
approximate the flow of $f$.

\subsection{Euler-type methods}

Using the notation established in \autoref{sec:prelim}, we now define
a family of ``Euler-type'' methods for conditionally linear
systems. This family includes the classical Euler method, as well as
the exponential Euler and SI Euler methods, among others.

\begin{definition}
  An \emph{Euler-type method} for \eqref{eqn:cl} has the form
  \begin{equation}
    \label{eqn:euler-type}
    x ^{ n + 1 } _i = \Phi ^{ (i) } _{ h, i } ( x ^n ) , \qquad i = 1, \ldots, {d} ,
  \end{equation}
  where each $ \Phi ^{ (i) } _h $ is either an exact or approximate
  time-$h$ flow for $ f ^{ ( i ) } $.  More explicitly, this flow has
  the form
  \begin{equation*}
    x _i ^{ n + 1 } = r _i \bigl( h a _i (x ^n ) \bigr) x _i ^n + \dfrac{ r _i \bigl( h a _i (x ^n ) \bigr) - 1 }{ h a _i (x ^n ) } h b _i (x ^n ) , \qquad i = 1 , \ldots, {d} , 
  \end{equation*}
  where $ r _i (z) = 1 + z + \mathcal{O} ( z ^2 ) $ is either the
  exponential function or an approximate exponential (e.g., the
  stability function of a Runge--Kutta method, a Pad\'e approximant,
  etc.).
\end{definition}

The key feature of these methods is that the nonlinear functions
$ a _i , b _i $ are only evaluated at $ x ^n $ when advancing
$ x ^n \mapsto x ^{ n + 1 } $. In other words, all of the components
are stepped forward in parallel.

\begin{example}[Euler's method]
  Suppose $ \Phi _h ^{ (i) } $ is Euler's method applied to
  $ f ^{ (i) } $ with time step size $h$, for $ i = 1 , \ldots, {d}
  $. Then \eqref{eqn:euler-type} becomes
  \begin{equation*}
    x _i ^{ n + 1 } = x _i ^n + h \bigl( a _i ( x ^n ) x _i ^n + b _i ( x ^n ) \bigr) , \qquad i = 1 , \ldots, {d}, 
  \end{equation*}
  which is just Euler's method applied to $f$.
\end{example}

\begin{example}[exponential Euler method]
  Suppose $ \Phi _h ^{ (i) } = \varphi _h ^{ (i) } $ is the exact
  time-$h$ flow of $ f ^{ (i) } $, for $ i = 1, \ldots, {d} $. Then
  \eqref{eqn:euler-type} becomes
  \begin{equation*}
    x _i ^{ n + 1 } = \exp \bigl( h a _i (x ^n ) \bigr) x _i ^n + \dfrac{ \exp \bigl( h a _i (x ^n )  \bigr) - 1 }{ h a _i (x ^n ) } h b _i (x ^n ), \qquad i = 1, \ldots, {d} .
  \end{equation*}
  This is the \emph{exponential Euler method}.
\end{example}

\begin{example}[SI Euler method]
  Suppose $ \Phi _h ^{ (i) } $ is the backward Euler method applied to
  $ f ^{ (i) } $ with time step size $h$, for $ i = 1 , \ldots, {d}
  $. Then \eqref{eqn:euler-type} becomes
  \begin{equation*}
    x _i ^{ n + 1 }  = \frac{ 1 }{ 1 - h a _i ( x ^n ) } x ^n _i + \frac{ 1 }{ 1 - h a _i ( x ^n ) } h b _i ( x ^n ) , \qquad i = 1, \ldots, {d} ,
  \end{equation*}
  which can also be written as
  \begin{equation*}
    x _i ^{ n + 1 } = x _i ^n + h \bigl( a _i ( x ^n ) x _i ^{ n + 1 } + b _i ( x ^n ) \bigr) , \qquad i = 1, \ldots, {d} .
  \end{equation*}
  This is the \emph{semi-implicit Euler} (or \emph{SI Euler})
  \emph{method}.
\end{example}

All of these methods have order $1$, and all reduce to Euler's method
in the special case that $ a _i = 0 $ for all $ i = 1, \ldots, {d}
$.

\subsection{The exponential midpoint method}

\citet{BoNe2013} also considered an explicit second-order method,
called the exponential midpoint method. Although this is neither an
Euler-type nor a splitting/composition method, we include it for
comparison.

\begin{definition}
  The \emph{exponential midpoint method} for \eqref{eqn:cl} is
  \begin{alignat*}{2}
    x _i ^{ n + 1/2 } &= \exp \bigl( \tfrac{1}{2} h a _i (x ^n )
    \bigr) x _i ^n + \dfrac{ \exp \bigl( \tfrac{1}{2} h a _i (x ^n )
      \bigr) - 1 }{ \tfrac{1}{2} h a _i (x ^n ) } \tfrac{1}{2} h b _i
    (x ^n ), &\qquad i &= 1, \ldots, {d} .\\
    x _i ^{ n + 1 } &= \exp \bigl( h a _i (x ^{n+1/2} ) \bigr) x _i ^n
    + \dfrac{ \exp \bigl( h a _i (x ^{n+1/2} ) \bigr) - 1 }{ h a _i (x
      ^{n+1/2} ) } h b _i (x ^{n+1/2} ), &\qquad i &= 1, \ldots, {d} .
  \end{alignat*}
\end{definition}

The second line is similar to exponential Euler, except $ a _i $ and
$ b _i $ are evaluated at the approximate midpoint $ x ^{ n + 1/2 } $,
obtained by first taking a half-step of exponential Euler. Note that
each nonlinear function is evaluated twice, first at $ x ^n $ and
again at $ x ^{ n + 1/2} $, so this method is twice as expensive per
step as an Euler-type method.

\begin{remark}
  This is an example of a two-stage explicit exponential Runge--Kutta
  method. One may also construct higher-order methods with more
  intermediate stages (and thus more function evaluations per
  step). See \citet{HoOs2010} for a survey of these and other
  exponential integrators.  Also, as before, one may replace the
  exponential function with an approximate exponential, but we will
  not consider that generalization here.
\end{remark}

\subsection{Splitting and composition methods}
\label{sec:splitting}

Before introducing the particular splitting and composition methods
that we propose for conditionally linear systems, we briefly review
what these classes of methods are, in general. For more on the general
theory and application of such methods, we refer the interested reader
to the survey by \citet{McQu2002}.

A splitting method, for an arbitrary dynamical system
$ \dot{x} = f (x) $, is based on the idea of decomposing (or
``splitting'') the vector field $f$ into a sum of vector fields
$ f = f ^{ (1) } + \cdots + f ^{ (m) } $ and approximating the
time-$h$ flow of $f$ by a composition of flows of the $ f ^{ (i) } $.
(Here, $f$ need not be conditionally linear, and it might even be a
vector field on a manifold.)  For example, if
$ f = f ^{ (1) } + f ^{ (2) } $, one might approximate the time-$h$
flow of $f$ by either of
\begin{equation*}
  \varphi ^{ (1) } _h \circ \varphi ^{ (2) } _h, \qquad \varphi ^{ (2) } _{ h / 2 } \circ \varphi ^{ (1) } _h \circ \varphi ^{ (2) } _{h/2} .
\end{equation*}
(One may also interchange $ \varphi ^{ (1) } $ and
$ \varphi ^{ (2) } $, but we prefer to think of that as corresponding
to the alternative splitting of $f$ that interchanges $ f ^{ (1) } $
and $ f ^{ (2) } $.) The first of these, called the \emph{Lie--Trotter
  splitting} (\citet{Trotter1959}) approximates the flow of $f$ with
order $1$; the second, known as the \emph{Strang splitting}
(\citet{Strang1968}), approximates it with order $2$. Although it
appears that the Strang splitting requires more function evaluations
than the Lie--Trotter splitting, note that we may use the semigroup
property
$ \varphi _h ^{ (2) } = \varphi _{ h / 2 } ^{ (2) } \circ \varphi _{ h
  / 2 } ^{ (2) } $ to write
\begin{multline*}
  (\varphi ^{ (2) } _{ h / 2 } \circ \varphi ^{ (1) } _h \circ \varphi ^{ (2) } _{h/2}   ) \circ \cdots \circ (\varphi ^{ (2) } _{ h / 2 } \circ \varphi ^{ (1) } _h \circ \varphi ^{ (2) } _{h/2}   )\\
  = \varphi ^{ (2) } _{ h / 2 } \circ ( \varphi _h ^{ (1) } \circ
  \varphi _h ^{ (2) } \circ \cdots \circ \varphi _h ^{ (2) } \circ
  \varphi _h ^{ (1) } ) \circ \varphi ^{ (2) } _{ h / 2 }.
\end{multline*}
Therefore, with the exception of the first and last step, both methods
simply alternate between the flows of $ f ^{ (1) } $ and
$ f ^{ (2) } $ and thus have essentially the same computational
cost. It is also possible to construct even higher-order splitting
methods, by alternating several times between fractional steps of
$ \varphi ^{ (1) } $ and $ \varphi ^{ (2) } $, as in the methods of
\citet{Yoshida1990} and \citet{Suzuki1990}.

Composition methods are just like splitting methods, except that the
flow of $ f ^{ ( i ) } $ may be replaced by an approximate
flow. Unlike exact flows, approximate flows do not generally form a
one-parameter group (or semigroup). This motivates the following
definitions, which will prepare us to discuss symmetric compositions
of approximate flows, analogous to the Strang splitting
method.

\begin{definition}
  The \emph{adjoint} of an approximate flow $ \Phi _h $ is defined to
  be $ \Phi ^\ast _h \coloneqq \Phi _{ - h } ^{-1} $. We say that
  $ \Phi _h $ is \emph{symmetric} if it is its own adjoint.
\end{definition}

\begin{remark}
  \label{rmk:rk_adjoint}
  If a Runge--Kutta method has the stability function $r$, then its
  adjoint method has the stability function
  $ r ^\ast (z) = 1/r(-z) $.
\end{remark}

Here are a few illustrative examples:
\begin{itemize}
\item If $ \varphi _h $ is the exact time-$h$ flow of a vector field,
  then the group property
  $ \varphi _h \circ \varphi _{ - h } = \mathrm{id} $ implies that
  $ \varphi _h $ is symmetric.

\item If $ \Phi _h $ is Euler's method, then $ \Phi _h ^\ast $ is the
  backward Euler method. The symmetric methods
  $ \Phi _{ h/2 } ^\ast \circ \Phi _{ h / 2 } $ and
  $ \Phi _{h/2} \circ \Phi _{h/2} ^\ast $ are the trapezoid and
  midpoint methods, respectively.

\item If $ \Phi _h = \varphi _h ^{ (1) } \circ \varphi ^{ (2) } _h $
  is the Lie--Trotter splitting method, then
  $ \Phi _h ^\ast = \varphi _h ^{ (2) } \circ \varphi ^{ (1) } _h $ is
  the Lie--Trotter method for the splitting with $ f ^{ (1) } $ and
  $ f ^{ (2) } $ interchanged. The symmetric methods
  $ \Phi _{ h / 2 } ^\ast \circ \Phi _{ h / 2 } $ and
  $ \Phi _{ h / 2 } \circ \Phi _{ h / 2 } ^\ast $ are the Strang splitting
  methods for the original splitting and for the splitting with
  $ f ^{ (1) } $ and $ f ^{ (2) } $ interchanged, respectively.
\end{itemize}

For instance, given a splitting $ f = f ^{ (1) } + f ^{ (2) } $, we may consider the composition methods
\begin{equation*}
  \Phi _h ^{ (1) } \circ \Phi _h ^{ (2) } , \qquad \Phi _{ h / 2 } ^{ (2) \ast } \circ \Phi _{ h / 2 } ^{ (1) \ast } \circ \Phi _{ h / 2 } ^{ (1) } \circ \Phi _{ h / 2 } ^{ (2) } ,
\end{equation*}
which are the composition-method generalizations of the Lie--Trotter
and Strang splitting methods, respectively. (We recover precisely
these splitting methods in the special case where we take
$ \Phi ^{ (1) } = \varphi ^{ (1) } $ and
$ \Phi ^{ (2) } = \varphi ^{ (2) } $ to be the exact flows.) In
particular, the first composition method has order $1$, while the
second is symmetric and has order $2$. As with splitting methods, one
may construct higher-order (symmetric) composition methods by
alternating several fractional steps of these approximate flows
(\citet{McLachlan1995}).

Before returning to conditionally linear systems, we finally note that
for $ f = f ^{ (1) } + \cdots + f ^{ (m) } $ with arbitrary $m$, we
may generalize the Lie--Trotter and Strang splitting and composition
methods above by the following:
\begin{equation*}
  \Phi _h ^{ (1) } \circ \cdots \circ \Phi _h ^{ (m) } , \qquad \Phi _{ h / 2 } ^{ (m) \ast } \circ \cdots \circ \Phi _{ h / 2 } ^{ (1) \ast } \circ \Phi _{ h / 2 } ^{ (1) } \circ \cdots \circ \Phi _{ h / 2 } ^{ (m) } .
\end{equation*}
Again, the first of these has order $1$, while the second is symmetric
and has order $2$. For more on these splitting and composition methods
and their higher-order generalizations, we again refer the reader to
the survey by \citet{McQu2002}.

\begin{definition}
  Given a conditionally linear system \eqref{eqn:cl}, we consider both
  the \emph{non-symmetric composition method}
  \begin{equation*}
    x ^{ n + 1 } = ( \Phi _h ^{ (1) } \circ \cdots \circ \Phi _h ^{ ({d})} ) ( x ^n ) ,
  \end{equation*}
  and the \emph{symmetric composition method}
  \begin{equation*}
    x ^{ n + 1 } = ( \Phi _{h/2} ^{ ({d}) \ast} \circ \cdots \circ \Phi _{ h/2 }^{ (1)\ast } \circ \Phi _{ h / 2 } ^{ (1) } \circ \cdots \circ  \Phi _{h/2} ^{ ({d})} ) ( x ^n ),
  \end{equation*}
  where each $ \Phi _h ^{ (i) } $ is either an exact or approximate
  time-$h$ flow for $ f ^{ (i) } $. We call these \emph{splitting
    methods} when all of the $ \Phi _h ^{ (i) } $ are exact flows.
\end{definition}

\begin{remark}
  As with Euler-type methods, each flow $ \Phi _h ^{ (i) } $ is
  defined by $ r _i (z) = 1 + z + \mathcal{O} (z ^2) $, which is an
  exact or approximate exponential. As in \autoref{rmk:rk_adjoint},
  its adjoint $ \Phi _h ^{ (i) \ast } $ corresponds to
  $ r _i ^\ast (z) = 1 / r _i ( - z ) $.
\end{remark}

\begin{remark}
  \label{rmk:commuting_flows}
  In the special case where the flows
  $ \Phi _h ^{ ( 1) } , \ldots, \Phi _h ^{ ({d}) } $ commute, the
  non-symmetric splitting method is identical to the corresponding
  Euler-type method. (This is true, for instance, if $ a _i , b _i $
  are constants.) In this case, the symmetric splitting method is
  similarly identical to the Euler-type method
  $ \Phi _{ h , i } = ( \Phi _{ h / 2 } ^{ (i) \ast } \circ \Phi _{ h
    / 2 } ^{ (i) } ) _i $.
\end{remark}

We now give examples of these methods in the $ {d} = 2 $ case, i.e.,
\begin{align*} 
  \dot{ x } _1 &= a _1 ( x _2 ) x _1 + b _1 (x _2 ) ,\\
  \dot{x} _2 &= a _2 (x _1 ) x _2 + b _2 (x _1 ) .
\end{align*} 
Here, we have made explicit in the notation that $ a _1 , b _1 $
depend only on $ x _2 $ and $ a _2 , b _2 $ depend only on $ x _1 $.

\begin{example}[symplectic Euler and St\"ormer/Verlet methods]
  \label{ex:composition}
  Suppose we take $ \Phi _h ^{ (1) } $ to be the time-$h$ flow of
  Euler's method and $ \Phi _h ^{ (2) } $ to be the time-$h$ flow of
  the backward Euler method.

  The non-symmetric composition method can be written as the algorithm
  \begin{align*}
    x _2 ^{ n + 1 } &= x _2 ^n + h \bigl( a _2 ( x _1 ^n ) x _2 ^{n+1} + b _2 ( x _1 ^n ) \bigr) ,\\
    x _1 ^{ n + 1 } &= x _1 ^n + h \bigl( a _1 ( x _2 ^{ n + 1 }  ) x _1 ^n + b _1 ( x _2 ^{n+1} ) \bigr) .
  \end{align*}
  This is an order-$1$ partitioned Runge--Kutta method known as the
  \emph{symplectic Euler method} (since its flow is symplectic when
  applied to Hamiltonian systems). Note that this is actually an
  explicit method, since the first step only requires solving a linear
  equation for $ x _2 ^{ n + 1 } $, as with the {SI} Euler method.

  The symmetric composition method can be written as the algorithm
  \begin{align*}
    x _2 ^{ n + 1/2 } &= x _2 ^n + \frac{1}{2}  h \bigl( a _2 ( x _1 ^n ) x _2 ^{n+1/2} + b _2 ( x _1 ^n ) \bigr) ,\\
    x _1 ^{ n + 1/2 } &= x _1 ^n + \frac{1}{2} h \bigl( a _1 ( x _2 ^{n+1/2} ) x _1 ^n + b _1 ( x _2 ^{n+1/2} ) \bigr) ,\\
    x _1 ^{ n + 1 } &= x _1 ^{n+1/2} + \frac{1}{2} h \bigl( a _1 ( x _2 ^{n+1/2} ) x _1 ^{n+1} + b _1 ( x _2 ^{n+1/2} ) \bigr) ,\\
    x _2 ^{ n + 1 } &= x _2 ^{n+1/2} + \frac{1}{2}  h \bigl( a _2 ( x _1 ^{n+1} ) x _2 ^{n+1/2} + b _2 ( x _1 ^{n+1} ) \bigr) .
  \end{align*}
  Note that the second and third steps can be combined, yielding
  \begin{align*}
    x _2 ^{ n + 1/2 } &= x _2 ^n + \frac{1}{2}  h \bigl( a _2 ( x _1 ^n ) x _2 ^{n+1/2} + b _2 ( x _1 ^n ) \bigr) ,\\
    x _1 ^{ n + 1 } &= x _1 ^n + h \biggl( a _1 ( x _2 ^{n+1/2} ) \frac{ x _1 ^n + x _1 ^{n+1} }{ 2 }  + b _1 ( x _2 ^{n+1/2} ) \biggr) ,\\
    x _2 ^{ n + 1 } &= x _2 ^{n+1/2} + \frac{1}{2}  h \bigl( a _2 ( x _1 ^{n+1} ) x _2 ^{n+1/2} + b _2 ( x _1 ^{n+1} ) \bigr) ,
  \end{align*}
  i.e., we use the fact that
  $ \Phi _{h/2} ^{ (1) \ast } \circ \Phi _{ h / 2 } ^{ (1) } $ is the
  time-$h$ flow of the trapezoid method for $ f ^{ (1) } $.  This is
  an order-$2$ symplectic partitioned Runge--Kutta method known as the
  \emph{St\"ormer/Verlet} (or \emph{leapfrog}) \emph{method}. As with
  the non-symmetric composition, this is actually an explicit method,
  since we need only solve linear equations for $ x _2 ^{ n + 1/2 } $
  and $ x _1 ^{ n + 1 } $.
\end{example}

\begin{example}[splitting methods]
  \label{ex:splitting}
  Suppose we take $ \Phi _h ^{ (1) } = \varphi _h ^{ (1) } $ and
  $ \Phi _h ^{ (2) } = \varphi _h ^{ (2) } $ to be the exact time-$h$
  flows of $ f ^{ (1) } $ and $ f ^{ (2) }$, respectively.

  The non-symmetric splitting method is
  \begin{align*}
    x _2 ^{ n + 1 } &= \exp \bigl( h a _2 (x _1 ^n ) \bigr) x _2 ^n + \dfrac{ \exp \bigl( h a _2 (x _1 ^n )  \bigr) - 1 }{ h a _2 (x _1 ^n ) } h b _2 (x _1 ^n ),\\
    x _1 ^{ n + 1 } &= \exp \bigl( h a _1 (x _2 ^{n+1} ) \bigr) x _1 ^n + \dfrac{ \exp \bigl( h a _1 (x _2 ^{n+1} )  \bigr) - 1 }{ h a _1 (x _2 ^{n+1} ) } h b _1 (x _2 ^{n+1} ),
  \end{align*}
  and the symmetric splitting method is
  \begin{align*}
    x _2 ^{ n + 1/2 } &= \exp \bigl( \tfrac{1}{2} h a _2 (x _1 ^n ) \bigr) x _2 ^n + \dfrac{ \exp \bigl( \tfrac{1}{2} h a _2 (x _1 ^n )  \bigr) - 1 }{ \tfrac{1}{2} h a _2 (x _1 ^n ) } \tfrac{1}{2} h b _2 (x _1 ^n ),\\
    x _1 ^{ n + 1 } &= \exp \bigl( h a _1 (x _2 ^{n+1/2} ) \bigr) x _1 ^n + \dfrac{ \exp \bigl( h a _1 (x _2 ^{n+1/2} )  \bigr) - 1 }{ h a _1 (x _2 ^{n+1/2} ) } h b _1 (x _2 ^{n+1/2} ),\\
    x _2 ^{ n + 1 } &= \exp \bigl( \tfrac{1}{2} h a _2 (x _1 ^{n+1} ) \bigr) x _2 ^{n+1/2} + \dfrac{ \exp \bigl( \tfrac{1}{2} h a _2 (x _1 ^{n+1} )  \bigr) - 1 }{ \tfrac{1}{2} h a _2 (x _1 ^{n+1} ) } \tfrac{1}{2} h b _2 (x _1 ^{n+1} ) .
  \end{align*}
  As previously stated, the non-symmetric splitting method has order
  $1$, whereas the symmetric splitting method has order $2$.
\end{example}

Note that all of the non-symmetric methods (and all of the symmetric
methods) agree in the special case where $ a _i = 0 $ for
$ i = 1 , \ldots, {d} $. For example, when $ {d} = 2 $ and
$ a _1 = a _2 = 0 $, the non-symmetric methods all reduce to
symplectic Euler, while the symmetric methods all reduce to
St\"ormer/Verlet.

\subsection{Combining Euler-type and splitting/composition methods}
\label{sec:euler-splitting}

We briefly discuss a generalization that includes both Euler-type and
splitting/composition methods, as well as methods combining aspects of
each. The idea is to partition the components $ \{ 1, \ldots, {d} \} $
and to apply an Euler-type method \emph{across} partitions while using
a splitting/composition method \emph{within} each partition. We will
not analyze these generalized methods in this paper, but we mention
them due to their parallel-implementation advantages when $ {d} $ is
large.

Let $ 0 = i _0 < \cdots < i _k = {d} $, where each $ i _j $ is an
integer.  This partitions $ \{ 1, \ldots, {d} \} $ into the $k$
subsets $ \{ i _{j-1} + 1 , \ldots, i _j \} $. We may then consider
the non-symmetric method
\begin{equation*}
  \Phi _{ h , i } = ( \Phi _h ^{ ( i _{ j -1 } + 1 ) } \circ \cdots \circ \Phi _h ^{ (i _j ) } ) _i , \qquad i = i _{ j -1 } + 1 , \ldots , i _j , \quad j = 1 , \ldots, k .
\end{equation*}
This is an Euler-type method in the special case when we partition
into $ {d} $ subsets of size $1$, and it is a non-symmetric
composition method when we partition into $1$ subset of size $ {d}
$. Likewise, we may consider
\begin{multline*}
  \Phi _{ h , i } = ( \Phi _{ h/2} ^{ (i _j ) \ast } \circ \cdots \circ \Phi _{ h / 2} ^{ (i_{j-1} + 1 ) \ast } \circ \Phi _{h/2} ^{ ( i _{ j -1 } + 1 ) } \circ \cdots \circ \Phi _{h/2} ^{ (i _j ) } ) _i ,\\
  i = i _{ j -1 } + 1 , \ldots , i _j , \quad j = 1 , \ldots, k ,
\end{multline*}
which is (respectively) an Euler-type method or a symmetric
composition method in the two special cases mentioned above.

While splitting and composition methods have desirable
structure-preserving properties (as we will see in the subsequent
sections), one disadvantage, compared to Euler-type methods, is that
the flows must be evaluated in series rather than in parallel. When
$ {d} $ is large, it may be computationally infeasible to do this. The
generalization above is a compromise that allows one to partition each
step into $k$ pieces that may be computed in parallel. For example,
one might simulate a network of $k$ Hodgkin--Huxley neurons in this
way.

\section{Limit cycle preservation for the Van~der~Pol oscillator}

\subsection{The Van~der~Pol oscillator}
\label{sec:vdp_intro}

The simple harmonic oscillator is a Hamiltonian system with
$ H ( x _1 , x _2 ) = \frac{1}{2} ( x _1 ^2 + x _2 ^2 ) $, so that
\begin{alignat*}{2}
  \dot{ x } _1 &= \hphantom{-}\frac{ \partial H }{ \partial x _2 } &&= \hphantom{-}x _2 ,\\
  \dot{ x } _2 &= -\frac{ \partial H }{ \partial x _1 } &&= -x _1 .
\end{alignat*}
The \emph{Van~der~Pol oscillator} is the conditionally linear system
\begin{equation}
  \label{eqn:vdp}
  \begin{aligned}
    \dot{x} _1 &= x _2 ,\\
    \dot{x} _2 &= \epsilon ( 1 - x _1 ^2 ) x _2 - x _1 ,
  \end{aligned}
\end{equation}
which adds a nonlinear dissipation term proportional to the constant
parameter $\epsilon$ (\citet{VanDerPol1926}). Due to its dynamical
similarities with the more complex Hodgkin--Huxley model, the Van der
Pol oscillator has played an important role in simplified models of
neuronal spiking: specifically, it is the foundation for the
Fitzhugh--Nagumo model \citep{Fitzhugh1961,NaArYo1962,Izhikevich2007},
of which it is a special case.

We begin by discussing the non-stiff case, when $ \epsilon \ll 1 $.
If we transform into the ``action-angle'' coordinates
$ ( a, \theta ) $, defined by $ x _1 = \sqrt{ 2 a } \cos \theta $,
$ x _2 = \sqrt{ 2 a } \sin \theta $, then \eqref{eqn:vdp} becomes
\begin{align*}
  \dot{ a } &= \epsilon ( 1 - 2 a \cos ^2 \theta ) 2 a \sin ^2 \theta ,\\
  \dot{ \theta } &= - 1 + \epsilon ( 1 - 2 a \cos ^2 \theta ) \cos \theta \sin \theta .
\end{align*}
When $ \epsilon \ll 1 $, we have $ \dot{ \theta } \approx - 1 $. Since
$ a $ evolves much more slowly, one may obtain approximate dynamics by
averaging over
$ \theta \in \mathbb{T} \coloneqq \mathbb{R} / 2 \pi \mathbb{Z} $,
i.e., over one period of oscillation. Observe that $ \sin ^2 \theta $
and $ 4 \cos ^2 \theta \sin ^2 \theta = \sin ^2 ( 2 \theta ) $ both
have average $ \frac{1}{2} $, so the averaged dynamics are given by
\begin{equation*}
  \dot{ a } \approx \epsilon ( a - \tfrac{1}{2} a ^2 ) = \epsilon a ( 1 - \tfrac{1}{2} a ) .
\end{equation*} 
This has fixed points $ a = 0 , 2$. When $ \epsilon < 0 $, the fixed
point at $ a = 0 $ is stable and the one at $ a = 2 $ is unstable,
whereas when $ \epsilon > 0 $, the fixed point at $ a = 0 $ is
unstable and the one at $ a = 2 $ is stable. In the
$ ( x _1 , x _2 ) $-plane, these equilibria correspond to a fixed
point at the origin and a limit cycle given by the circle with radius
$2$ centered at the origin, and the bifurcation at $ \epsilon = 0 $ is
a Hopf bifurcation.

The stiff case, when $ \epsilon \gg 1 $, is most easily understood
after performing the change of variables $ y _1 = x _1 $,
$ y _2 = x _1 - x _1 ^3 / 3 - x _2 / \epsilon $, known as the
\emph{Li\'enard transformation} \citep{Lienard1928}. The system
\eqref{eqn:vdp} then becomes
\begin{align*}
  \dot{ y } _1 &= \epsilon ( y _1 - y _1 ^3 / 3 - y _2 ) ,\\
  \dot{ y } _2 &= - y _1 / \epsilon .
\end{align*}
The $ y _1 $-nullcline (i.e., where $ \dot{ y } _1 = 0 $) is given by
the cubic $ y _2 = y _1 - y _1 ^3 / 3 $. Since $ y _1 $ evolves much
more quickly than $ y _2 $, solutions are quickly attracted to the
cubic nullcline. They then move slowly along the nullcline until they
reach an extremum, at which point they fall off the nullcline and
quickly jump horizontally to the other branch of the nullcline. This
repeats periodically, describing the attractive limit cycle of the
stiff Van~der~Pol oscillator.

Reference solutions for the Van~der~Pol oscillator, in both the
non-stiff and stiff cases, are shown in \autoref{fig:vdp_reference}.

\begin{figure}
  \centering
  \includegraphics{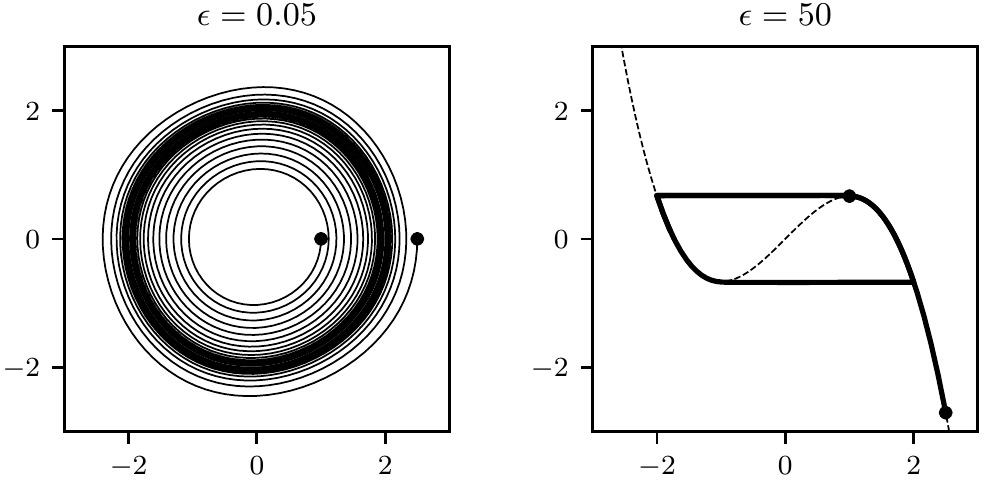}
  \caption{Reference solutions for the Van~der~Pol
    oscillator. \emph{Left:} When $ 0 < \epsilon \ll 1 $, solutions
    are attracted to a limit cycle with approximate radius $2$ in the
    $ ( x _1 , x _2 ) $-plane. \emph{Right:} When $ \epsilon \gg 1 $,
    the attractive limit cycle jumps between branches of the cubic
    nullcline (dashed line) in the $ ( y _1 , y _2 ) $-plane, given by
    the Li\'enard transformation
    $ ( y _1, y _2 ) = ( x _1, x _1 - x _1 ^3 / 3 - x _2 / \epsilon )
    $.\label{fig:vdp_reference}}
\end{figure}

\subsection{Limit cycle behavior of numerical methods (non-stiff
  case)}

We now analyze the limit cycle behavior of Euler-type methods, the
exponential midpoint method, and splitting/composition methods for
the Van~der~Pol oscillator in the non-stiff case $ \epsilon \ll 1
$. Following the approach of \citet{HaLu1999} (see also \citet[Chapter
XII]{HaLuWa2006}), we do so using \emph{backward error analysis}. That
is, we view a numerical method for the vector field $f$ as formally
the flow of a modified vector field
$ \widetilde{ f } = f + h f ^{ [1] } + h ^2 f ^{ [2] } + \cdots
$. This modified vector field is calculated by writing the numerical
method as $ x ^{ n + 1 } = \widetilde{ x } ( t ^n + h ) $, where
$ \widetilde{ x } $ formally solves the initial value problem
$ \dot{ \widetilde{ x } } = \widetilde{ f } ( \widetilde{ x } ) $,
$ \widetilde{ x } ( t ^n ) = x ^n $, and matching terms in the Taylor
expansion.

\begin{remark}
  By ``formally,'' we mean that $ \widetilde{ x } $ and
  $ \widetilde{ f } $ are formal power series, which may diverge.
  However, this formal procedure may be interpreted rigorously by
  truncating the asymptotic expansions and proving suitable error
  estimates; see \citet[Chapter IX]{HaLuWa2006}.
\end{remark}

\subsubsection{Euler-type methods}
Euler's method for the Van~der~Pol oscillator \eqref{eqn:vdp} is
\begin{align*}
  x _1 ^{ n + 1 } &= x _1 ^n + h x _2 ^n ,\\
  x _2 ^{ n + 1 } &= x _2 ^n + h \Bigl( \epsilon \bigl( 1 - (x _1 ^n)^2 \bigr) x _2 ^n - x _1 ^n \Bigr) .
\end{align*}
To calculate the modified vector field $ \widetilde{ f } $, we write
the modified system
\begin{align*}
  \dot{ \widetilde{ x } } _1 &= \widetilde{ x } _2  + h f _1 ^{ [1] } ( \widetilde{ x } ) + \mathcal{O} ( h ^2 ) ,\\
  \dot{ \widetilde{ x } } _2 &= \epsilon (  1 - \widetilde{ x } _1 ^2 ) \widetilde{ x } _2 - \widetilde{ x } _1 + h f ^{ [1] } _2 ( \widetilde{ x } ) + \mathcal{O} ( h ^2 ) .
\end{align*}
Taylor expanding the first component gives
\begin{align*}
  \widetilde{ x } _1 ( t ^n + h )
  &= \widetilde{ x } _1 ( t _n ) + h \dot{ \widetilde{ x } } _1 ( t _n ) + \frac{1}{2} h ^2 \ddot{ \widetilde{ x } } _1 ( t _n ) + \mathcal{O} (h ^3 )\\
  &= x _1 ^n + h \widetilde{ x } _2 + h ^2 \Bigl(  f _1 ^{ [1] } ( x ^n ) - \frac{1}{2} x _1 ^n + \mathcal{O} ( \epsilon ) \Bigr) + \mathcal{O} ( h ^3 ) .
\end{align*}
Likewise, for the second component,
\begin{equation*}
  \widetilde{ x } _2 ( t ^n + h ) = x _2 ^n + h \Bigl( \epsilon \bigl( 1 - (x _1 ^n ) ^2 \bigr) x _2 ^n - x _1 ^n \Bigr) + h ^2 \Bigl( f _2 ^{ [1] } ( x ^n ) - \frac{1}{2} x _2 ^n + \mathcal{O} (\epsilon) \Bigr) + \mathcal{O} ( h ^3 ) .
\end{equation*}
Matching terms with the expressions for $ x ^{ n + 1 } $ implies
$ f _1 ^{ [1] } ( \widetilde{ x } ) = \frac{1}{2} \widetilde{ x } _1 +
\mathcal{O} (\epsilon) $ and
$ f _2 ^{ [1] } ( \widetilde{ x } ) = \frac{1}{2} \widetilde{ x } _2 +
\mathcal{O} (\epsilon) $, so Euler's method is formally the flow of
the modified system
\begin{equation}
\begin{aligned}
  \dot{ \widetilde{ x } } _1 &= \widetilde{ x } _2 + \frac{1}{2} h \widetilde{ x } _1 +  \mathcal{O} ( h ^2 + \epsilon h ) ,\\
  \dot{ \widetilde{ x } } _2 &= \epsilon ( 1 - \widetilde{ x } _1 ^2 ) \widetilde{ x } _2 - \widetilde{ x } _1 + \frac{1}{2} h \widetilde{ x } _2 + \mathcal{O} ( h ^2 + \epsilon h ) .
\end{aligned}\label{eqn:eulerMVF}
\end{equation}
Transforming into action-angle coordinates and averaging over one
period of oscillation, as in \autoref{sec:vdp_intro}, gives
\begin{align*}
  \dot{ \widetilde{ a } } &\approx \epsilon ( \widetilde{ a } - \tfrac{1}{2} \widetilde{ a } ^2 ) + h \widetilde{ a } + \mathcal{O} ( h ^2 + \epsilon h ) \\
  &= \epsilon \widetilde{ a } ( 1 + \tfrac{h}{\epsilon} - \tfrac{1}{2} \widetilde{ a } ) + \mathcal{O} ( h ^2 + \epsilon h ) .
\end{align*}
This has an equilibrium at
$ \widetilde{ a } = 2 ( 1 + h / \epsilon ) $, so the corresponding
limit cycle in the
$ ( \widetilde{ x } _1 , \widetilde{ x } _2 ) $-plane is a circle of
radius $ 2 \sqrt{ 1 + h / \epsilon } $ centered at the origin.  This
gives a poor approximation of the true limit cycle, which has radius
$2$, unless $ h \ll \epsilon $. As noted in \autoref{sec:intro}, this
step size requirement is even more restrictive than that needed for
numerical stability. The foregoing argument appears in
\citet{HaLu1999} and in \citet[Chapter XII]{HaLuWa2006}.

We next show that \emph{all} Euler-type methods, including the
exponential Euler and SI Euler methods, share this poor limit cycle
behavior.

\begin{proposition}
  \label{prop:vdp_euler}
  For any Euler-type method applied to the Van~der~Pol oscillator, the
  modified vector field is given by
  \eqref{eqn:eulerMVF}. Consequently, numerical solutions have a limit
  cycle with approximate radius $ 2 \sqrt{ 1 + h / \epsilon } $ for
  $ \epsilon \ll 1 $.
\end{proposition}

\begin{proof}
  Since $ a _1 = 0 $, the first component of any Euler-type method
  agrees with Euler's method, i.e.,
  \begin{equation*}
    x _1 ^{ n + 1 } = x _1 ^n + h x _2 ^n .
  \end{equation*}
  For the second component, using
  $ r _2 (z) = 1 + z + \mathcal{O} ( z ^2 ) $ and
  $ \bigl( r _2 (z) - 1 \bigr) / z = 1 + \mathcal{O} (z) $, we have
  \begin{align*}
    x _2 ^{ n + 1 }
    &= \Bigl(  1 + h \epsilon \bigl( 1 - (x _1 ^n) ^2 \bigr) + \mathcal{O} ( \epsilon ^2 h ^2 ) \Bigr) x _2 ^n - h \bigl( 1 + \mathcal{O} ( \epsilon h ) \bigr) x _1 ^n \\
    &= x _2 ^n + h \Bigl( \epsilon \bigl( 1 - ( x _1 ^n ) ^2 \bigr) x _2 ^n - x _1 ^n \Bigr) + \mathcal{O} ( \epsilon h ^2 ) .
  \end{align*}
  Hence, $ f ^{ [1] } $ for an Euler-type method only differs from
  that for Euler's method by $ \mathcal{O} (\epsilon) $, which becomes
  $ \mathcal{O} ( \epsilon h ) $ in the modified vector
  field. However, this is just absorbed by the error term in
  \eqref{eqn:eulerMVF}.
\end{proof}

\subsubsection{The exponential midpoint method}

As with the Euler-type methods, we may also use backward error
analysis to analyze the limit cycle behavior of the exponential
midpoint method. Since the method is second-order, the first-order
term $ f ^{ [1] } $ in the modified vector field vanishes, and we have
\begin{align*}
  \dot{ \widetilde{ x } } _1 &= \widetilde{ x } _2  + h ^2 f _1 ^{ [2] } ( \widetilde{ x } ) + \mathcal{O} ( h ^3 ) ,\\
  \dot{ \widetilde{ x } } _2 &= \epsilon (  1 - \widetilde{ x } _1 ^2 ) \widetilde{ x } _2 - \widetilde{ x } _1 + h ^2 f ^{ [2] } _2 ( \widetilde{ x } ) + \mathcal{O} ( h ^3 ) .
\end{align*}
Na\"ively, one would expect these $ \mathcal{O} ( h ^2 ) $ errors to
result in an $ \mathcal{O} ( h ^2 / \epsilon ) $ error in the limit
cycle radius. However, Taylor expanding and matching terms yields,
after a calculation,
\begin{align*}
  \dot{ \widetilde{ x } } _1 &= \widetilde{ x } _2  + \frac{ 1 }{ 6 } h ^2 \widetilde{ x } _2 + \mathcal{O} ( h ^3 + \epsilon h ^2 )  ,\\
  \dot{ \widetilde{ x } } _2 &= \epsilon (  1 - \widetilde{ x } _1 ^2 ) \widetilde{ x } _2 - \widetilde{ x } _1 - \frac{ 1 }{ 6 } h ^2 \widetilde{ x } _1 + \mathcal{O} ( h ^3 + \epsilon h ^2 ) ,
\end{align*}
so these $ h ^2 $ terms actually cancel in
$ \dot{ \widetilde{ a } } = \widetilde{ x } _1 \dot{ \widetilde{ x } }
_1 + \widetilde{ x } _2 \dot{ \widetilde{ x } } _2 $. Therefore, we
need to compute out to $ f ^{ [3] } $ to obtain the leading-order
error term for the limit cycle radius.

\begin{proposition}
  \label{prop:vdp_expmid}
  For the exponential midpoint method applied to the Van~der~Pol
  oscillator, numerical solutions have a limit cycle with approximate
  radius $ 2 \sqrt{ 1 + h ^3 / ( 4 \epsilon ) } $ for
  $ \epsilon \ll 1 $.
\end{proposition}

\begin{proof}
  Taylor expanding and matching terms to compute $ f ^{[3]} $ gives
  \begin{align*}
    \dot{ \widetilde{ x } } _1 &= \widetilde{ x } _2  + \frac{ 1 }{ 6 } h ^2 \widetilde{ x } _2 + \frac{ 1 }{ 8 } h ^3 \widetilde{ x } _1 + \mathcal{O} ( h ^4 + \epsilon h ^2 )  ,\\
    \dot{ \widetilde{ x } } _2 &= \epsilon (  1 - \widetilde{ x } _1 ^2 ) \widetilde{ x } _2 - \widetilde{ x } _1 - \frac{ 1 }{ 6 } h ^2 \widetilde{ x } _1 + \frac{ 1 }{ 8 } h ^3 \widetilde{ x } _2 + \mathcal{O} ( h ^4 + \epsilon h ^2 ) ,
  \end{align*}
  Transforming into action-angle coordinates, the $ h ^2 $ terms
  cancel in
  $ \dot{ \widetilde{ a } } = \widetilde{ x } _1 \dot{ \widetilde{ x }
  } _1 + \widetilde{ x } _2 \dot{ \widetilde{ x } } _2 $, as noted
  above, while the $ h ^3 $ terms become
  $ \frac{ 1 }{ 8 } h ^3 ( \widetilde{ x } _1 ^2 + \widetilde{ x } _2
  ^2 ) = \frac{ 1 }{ 4 } h ^3 \widetilde{ a } $. Therefore, averaging
  over one period of oscillation gives
  \begin{equation*}
    \dot{ \widetilde{ a } } \approx \epsilon \widetilde{ a } ( 1 - \tfrac{1}{2} \widetilde{ a } ) + \tfrac{ 1 }{ 4 } h ^3 \widetilde{ a } = \epsilon \widetilde{ a } ( 1 + \tfrac{ h ^3 }{ 4 \epsilon } - \tfrac{1}{2} \widetilde{ a }  ),
  \end{equation*}
  which has an equilibrium at
  $ \widetilde{ a } = 2 \bigl( 1 + h ^3 / (4\epsilon)\bigr) $. The
  corresponding limit cycle in the
  $ ( \widetilde{ x } _1 , \widetilde{ x } _2 ) $-plane is a circle of
  radius $ 2 \sqrt{ 1 + h ^3 / (4 \epsilon) } $ centered at the
  origin.
\end{proof}

It follows that, in order to obtain a good approximation of the limit
cycle using the exponential midpoint method, we require
$ h ^3 \ll \epsilon $. This allows for larger time steps than
Euler-type methods, which require $ h \ll \epsilon $, but still we
cannot choose $h$ independently of $\epsilon$.

\subsubsection{Splitting and composition methods}
We next examine the limit cycle behavior of splitting and composition
methods for the Van~der~Pol oscillator. Instead of explicitly
computing the modified vector field, we exploit the fact that the
modified vector field is Hamiltonian when $ \epsilon = 0 $. This is an
application of a general argument for perturbed Hamiltonian systems
due to \citet{Stoffer1998,HaLu1999}. We remark that those authors were
primarily considering symplectic integrators, such as symplectic
(partitioned) Runge--Kutta methods, which are symplectic when applied
to \emph{any} Hamiltonian system. Although the splitting and
composition methods we consider are not symplectic in this more
general sense---they are generally non-symplectic for non-separable
Hamiltonian systems---the argument only requires that the modified
vector field be Hamiltonian when $ \epsilon = 0 $, which it is in this
case.

When $ \epsilon = 0 $, the Van~der~Pol oscillator reduces to the
simple harmonic oscillator. In this case, the splitting
$ f = f ^{ (1) } + f ^{ (2) } $ corresponds to the \emph{Hamiltonian
  splitting} $ H = H ^{ (1) } + H ^{ ( 2 ) } $, i.e., $ f ^{ (i) } $
is the Hamiltonian vector field for
$ H ^{ (i) } (x) = \frac{1}{2} x _i ^2 $. Since $ a _i = 0 $, any
approximate flow $ \Phi _h ^{ (i) } $ coincides with the exact flow
$ \varphi _h ^{ (i) } $, so for the simple harmonic oscillator, every
composition method is just a splitting method.

Since vector fields form a Lie algebra with the Jacobi--Lie bracket,
the modified vector field $ \widetilde{ f } $ for the non-symmetric
splitting method $ \varphi _h ^{ (1) } \circ \varphi _h ^{ (2) } $ may
be computed by applying the Baker--Campbell--Hausdorff formula to the
vector fields $ f ^{ (1) } , f ^{ (2) } $. Moreover, since Hamiltonian
vector fields are closed with respect to the Jacobi--Lie bracket
(i.e., they form a Lie subalgebra),
$ \widetilde{ f } = f + h f ^{ [1] } + \cdots $ is formally the
Hamiltonian vector field of a modified Hamiltonian
$ \widetilde{ H } = H + h H ^{ [1] } + \cdots $. In fact, the modified
Hamiltonian can itself be computed using the
Baker--Campbell--Hausdorff formula, by applying it to
$ H ^{ (1) } , H ^{ (2) } $ with the Poisson bracket; this works
because the Lie algebra of Hamiltonian vector fields with the
Jacobi--Lie bracket is isomorphic to that of Hamiltonian functions
with the Poisson bracket.

This approach is due to \citet{Yoshida1993}, and it may be generalized
to show that any Hamiltonian splitting method, including symmetric and
higher-order splittings, has a modified vector field that is again
Hamiltonian. Moreover, when the splitting method has order $p$, we
have $ \widetilde{ f } = f + \mathcal{O} ( h ^p ) $ and thus
$ \widetilde{ H } = H + \mathcal{O} ( h ^p ) $ (\citet[Theorem
IX.1.2]{HaLuWa2006}).

\begin{proposition}
  \label{prop:vdp_composition}
  Suppose we apply an order-$p$ composition method, based on the
  splitting $ f = f ^{ (1) } + f ^{ (2) } $, to the Van~der~Pol
  oscillator. Then numerical solutions have a limit cycle with
  approximate radius $ 2 + \mathcal{O} ( h ^p ) $ for
  $ \epsilon \ll 1 $. More generally, for $\epsilon$ sufficiently
  small (but not necessarily $ \ll 1 $), the numerical limit cycle is
  within $ \mathcal{O} ( h ^p ) $ of the exact limit cycle.
\end{proposition}

\begin{proof}
  Since the method has order $p$, it is formally the flow of the modified system
  \begin{align*}
    \dot{ \widetilde{ x } } _1 &= \widetilde{ x } _2 + \mathcal{O} ( h ^p ) ,\\
    \dot{ \widetilde{ x } } _2 &=  \epsilon ( 1 - \widetilde{ x } _1 ^2 ) \widetilde{ x } _2 - \widetilde{ x } _1 + \mathcal{O} ( h ^p ) .
  \end{align*}
  When $ \epsilon = 0 $, this is formally the flow of a modified
  Hamiltonian
  $ \widetilde{ H } ( x _1 , x _2 ) = \frac{1}{2} ( x _1 ^2 + x _2 ^2
  ) + \mathcal{O} ( h ^p ) $. Now, transforming into action-angle
  variables gives
  \begin{align*}
    \dot{ \widetilde{ a } } &= \epsilon ( 1 - 2 \widetilde{ a } \cos ^2 \widetilde{ \theta } ) 2 \widetilde{ a } \sin ^2 \widetilde{ \theta } + \mathcal{O} ( h ^p ) ,\\
    \dot{ \widetilde{ \theta } } &= - 1 + \epsilon ( 1 - 2 \widetilde{ a } \cos ^2 \widetilde{ \theta } ) \cos \widetilde{ \theta } \sin \widetilde{ \theta } + \mathcal{O} ( h ^p ) .
  \end{align*}
  Since the transformation $ ( x _1 , x _2 ) \mapsto ( a, \theta ) $
  is symplectic, it follows that the transformed flow is also
  Hamiltonian when $ \epsilon = 0 $, with
  $ \widetilde{ H } (a, \theta) = a + \mathcal{O} ( h ^p ) $.

  Now, this modified Hamiltonian flow contains all the terms not
  involving $\epsilon$, so we may write
  \begin{equation*}
    \dot{ \widetilde{ a } } = \frac{ \partial \widetilde{ H } }{ \partial \widetilde{ \theta } } + \epsilon ( 1 - 2 \widetilde{ a } \cos ^2 \widetilde{ \theta } ) 2 \widetilde{ a } \sin ^2 \widetilde{ \theta } + \mathcal{O} ( \epsilon h ^p ) .
  \end{equation*}
  However,
  $ \int _{ \mathbb{T} } (\partial \widetilde{ H } / \partial
  \widetilde{ \theta }) \,\mathrm{d} \widetilde{ \theta } = 0 $, so
  the terms not involving $\epsilon$ drop out when averaging over
  $ \theta \in \mathbb{T} $. This leaves the averaged dynamics
  \begin{equation*}
    \dot{ \widetilde{ a } } \approx \epsilon \widetilde{ a } ( 1 - \tfrac{1}{2} \widetilde{ a } ) + \mathcal{O} ( \epsilon h ^p ) = \epsilon \widetilde{ a } \bigl( 1 + \mathcal{O} ( h ^p ) - \tfrac{1}{2} \widetilde{ a } \bigr),
  \end{equation*} 
  which has an equilibrium at
  $ \widetilde{ a } = 2 + \mathcal{O} ( h ^p ) $, corresponding to a
  limit cycle with radius
  $ \sqrt{ 2 \widetilde{a} } = \sqrt{ 4 + \mathcal{O} ( h ^p ) } = 2 + \mathcal{O}
  ( h ^p ) $.

  Finally, the more general statement that the limit cycle is
  preserved up to $ \mathcal{O} ( h ^p ) $ follows from an argument of
  \citet{Stoffer1998,HaLu1999} (see also \citet[Theorem
  XII.5.2]{HaLuWa2006}), which accounts for the averaging
  approximation error when $\epsilon$ is sufficiently small but not
  necessarily $ \ll 1 $.
\end{proof}

\begin{corollary}
  \label{cor:vdp_composition}
  Consider the non-symmetric and symmetric composition methods of
  \autoref{sec:splitting} applied to the Van~der~Pol oscillator. For
  $\epsilon$ sufficiently small, the non-symmetric methods preserve
  the limit cycle up to $ \mathcal{O} ( h ) $, while the symmetric
  methods preserve it up to $ \mathcal{O} ( h ^2 ) $. In particular,
  when $ \epsilon \ll 1 $, the radius of the numerical limit cycle is
  $ 2 + \mathcal{O} (h) $ for the non-symmetric methods and
  $ 2 + \mathcal{O} ( h ^2 ) $ for the symmetric methods.
\end{corollary}

These results say that splitting/composition methods accurately
preserve the limit cycle of the Van~der~Pol oscillator when
$ h \ll 1 $. This allows for much larger step sizes than Euler-type
methods, which require $ h \ll \epsilon $, or the exponential midpoint
method, which requires $ h ^3 \ll \epsilon $.

\subsection{Numerical experiments}

In this section, we show the results of numerical experiments for the
methods of the previous section applied to the Van~der~Pol
oscillator. In the non-stiff case, we observe superior numerical limit
cycle preservation at large time steps for the splitting/composition
methods, compared to the Euler-type methods and (to a lesser extent)
the exponential midpoint method, which is consistent with the
theoretical results of the previous section. In the stiff case, the
splitting methods preserve the limit cycle behavior best, followed by
the composition methods and exponential midpoint method, with the
Euler-type methods performing worst, although we do not yet have a
theoretical explanation for this.

\subsubsection{Non-stiff Van~der~Pol oscillator}

\begin{figure}
  \centering
  \includegraphics[width=\textwidth]{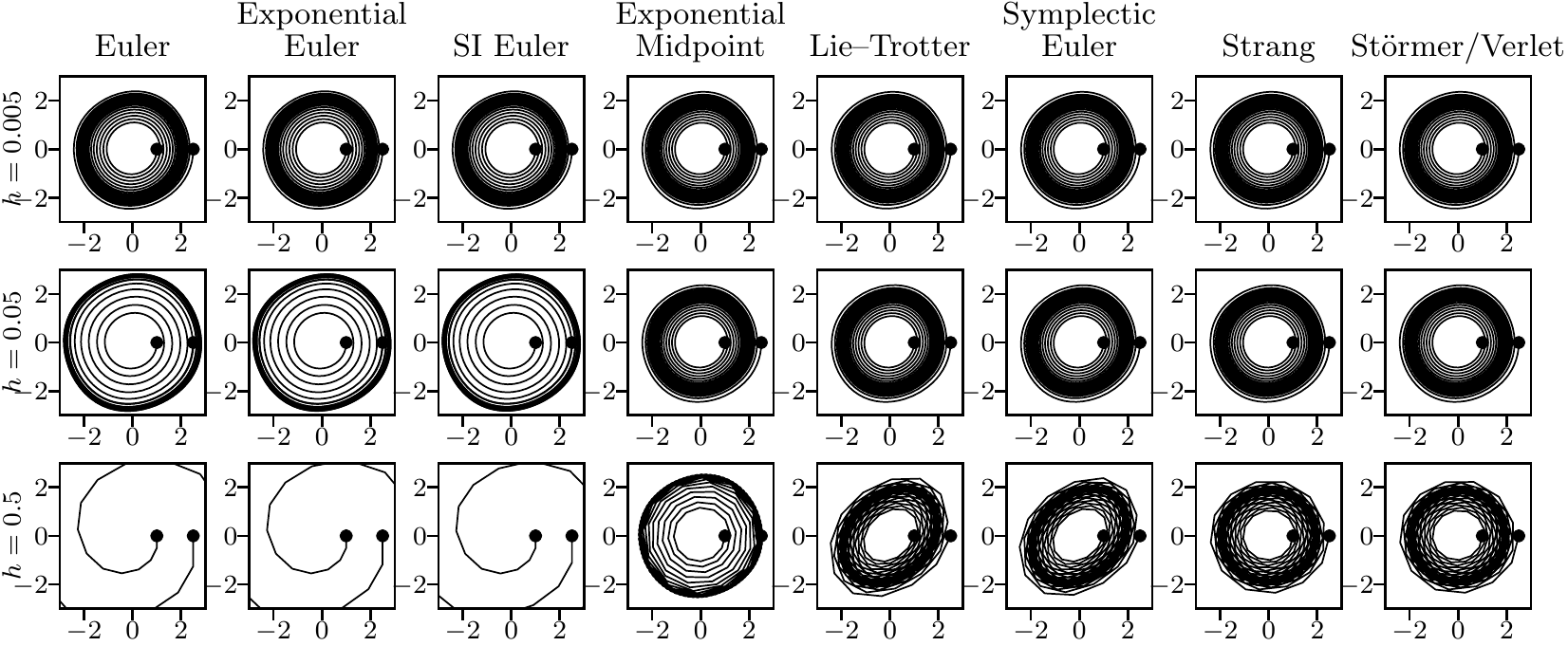}
  \caption{Numerical limit-cycle behavior ($ x _2 $ vs.~$ x _1 $) for
    the Van~der~Pol oscillator with $ \epsilon = 0.05 $. For the
    Euler-type methods, the limit cycle radius grows rapidly with $h$,
    consistent with \autoref{prop:vdp_euler}. For the exponential
    midpoint method, this growth is negligible at smaller time steps
    but is clearly visible at $ h = 0.5 $, consistent with
    \autoref{prop:vdp_expmid}. The splitting and composition methods
    exhibit much better limit-cycle preservation, consistent with
    \autoref{cor:vdp_composition}; some skewing is visible for the
    first-order, non-symmetric Lie--Trotter and symplectic Euler
    methods.\label{fig:vdp_nonstiff}}
\end{figure}

\autoref{fig:vdp_nonstiff} shows phase portraits for three Euler-type
methods (Euler, exponential Euler, and SI Euler), the exponential
midpoint method, and four splitting/composition methods (Lie--Trotter
splitting, symplectic Euler, Strang splitting, and St\"ormer/Verlet),
applied to the Van~der~Pol oscillator with $ \epsilon = 0.05 $. For
the Euler-type methods, the numerical limit cycle radius is seen to
grow with $h/\epsilon$, which is consistent with
\autoref{prop:vdp_euler}. For the exponential midpoint method, the
$ h ^3 / (4 \epsilon ) $ growth in limit cycle radius is less apparent
at smaller time steps but is clearly visible at $ h = 0.5 $, for
which $ h ^3 / ( 4 \epsilon ) = 0.625 $, consistent with
\autoref{prop:vdp_euler}. By contrast, for the non-symmetric and
symmetric splitting and composition methods, there is no apparent
growth in limit cycle radius---even when $ h = 0.5 $, for which
$ h/\epsilon = 10 $---which is consistent with
\autoref{cor:vdp_composition}. Notice that the asymmetry and lower
order of the Lie--Trotter splitting and symplectic Euler methods
manifests as a skewing of the limit cycle for large $h$, while the
limit cycle remains approximately circular for the symmetric methods.

\begin{figure}
  \centering
  \includegraphics{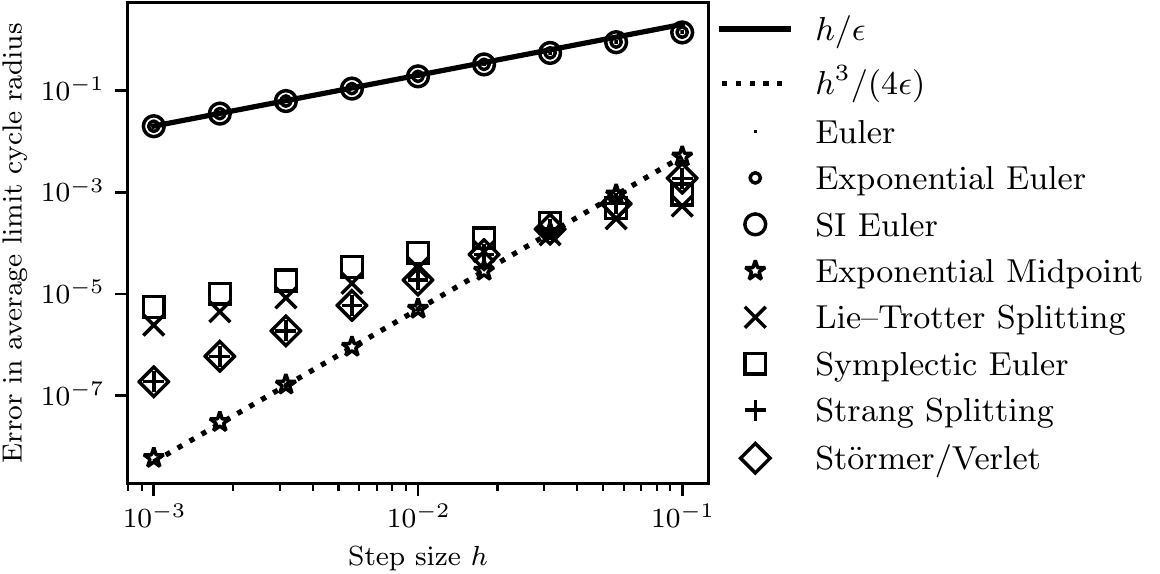}
  \caption{Error in average limit cycle radius for the Van~der~Pol
    oscillator with $ \epsilon = 0.05 $ and various step sizes $h$. We
    see that the Euler-type methods have error $ \sim h / \epsilon $,
    the exponential midpoint method has error
    $ \sim h ^3 / (4 \epsilon ) $, the non-symmetric
    splitting/composition methods have error $ \mathcal{O} (h) $, and
    the symmetric splitting/composition methods have error
    $ \mathcal{O} ( h ^2 ) $, consistent with the theoretical
    results.\label{fig:vdp_limit_cycle_convergence}}
\end{figure}

\autoref{fig:vdp_limit_cycle_convergence} illustrates how the average
limit cycle radius converges, for each of these methods, as
$ h \rightarrow 0 $. For the Euler-type methods, following
\autoref{prop:vdp_euler}, the error in average limit cycle radius is
$ \sim h / \epsilon $.  For the exponential midpoint method, following
\autoref{prop:vdp_expmid}, the error is $ \sim h ^3 / (4 \epsilon ) $.
For the splitting and composition methods, following
\autoref{cor:vdp_composition}, the error is $ \mathcal{O} (h) $ for
the non-symmetric methods and $ \mathcal{O} ( h ^2 ) $ for the
symmetric methods.

Although the exponential midpoint method has the lowest error for
small time steps, we make two remarks about this. First, this occurs
because $ h ^3 / \epsilon \ll h ^2 $ when $ h \ll \epsilon $, which is
more difficult to achieve when $\epsilon$ is even smaller. Second,
since it requires twice as many function evaluations as the other
methods, using the same computational budget would require time steps
twice as large, roughly increasing the error by a constant factor of
$8$.

\subsubsection{Stiff Van~der~Pol oscillator}
\label{sec:vdp_stiff}

We next apply these methods to the Van~der~Pol oscillator with
$ \epsilon = 50 $. In contrast with the non-stiff case, we see notably
different behavior among the three Euler-type methods, as well as
between the splitting and composition methods of the same order.

\begin{figure}
  \centering
  \includegraphics[width=\textwidth]{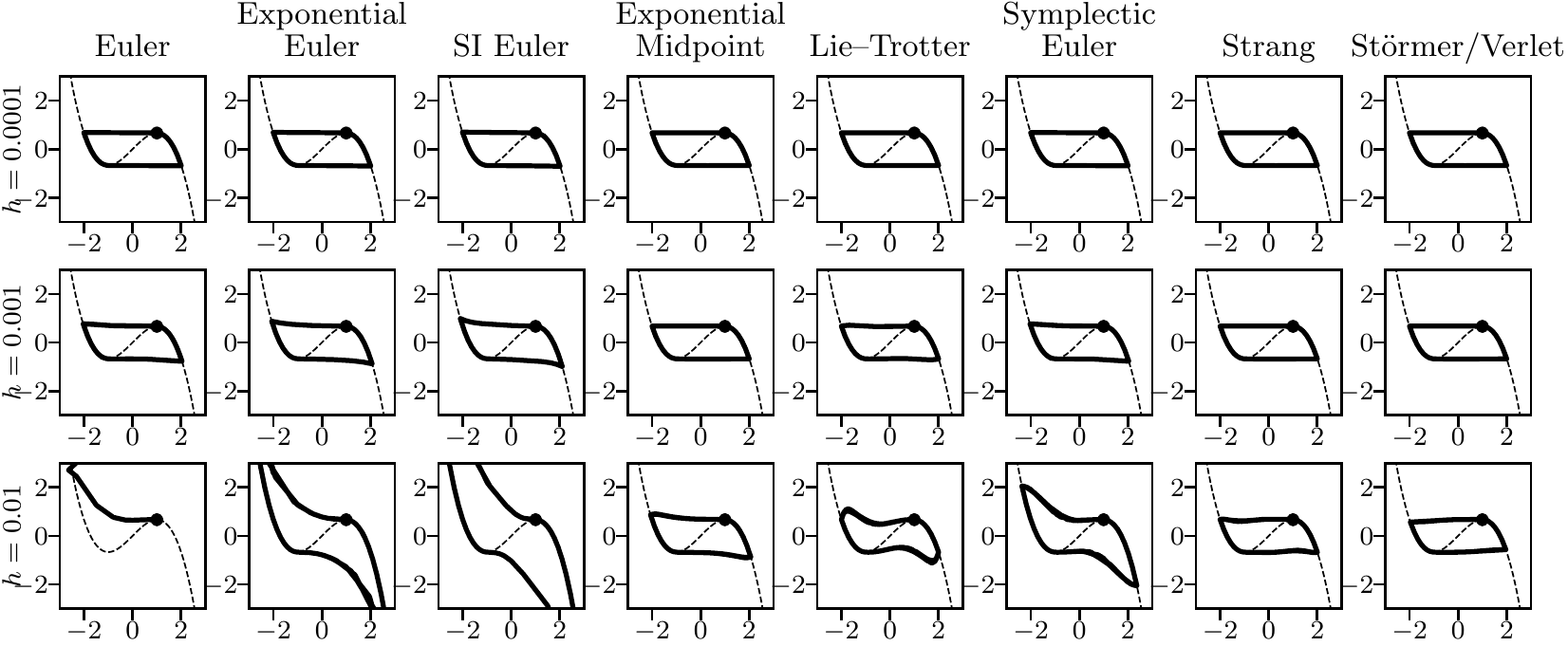}
  \caption{Numerical limit-cycle behavior ($ y _2 $ vs.~$ y _1 $,
    cubic nullcline shown as dashed line) for the Van~der~Pol
    oscillator with $ \epsilon = 50 $. As the time step size grows,
    the exponential midpoint and splitting/composition methods exhibit
    substantially less limit cycle distortion than the Euler-type
    methods (while Euler's method itself becomes unstable), and the
    Lie--Trotter splitting method performs notably better than the symplectic
    Euler composition method of the same order. \label{fig:vdp_stiff}}
\end{figure}

\autoref{fig:vdp_stiff} shows numerical phase plots in the
$ ( y _1 , y _2 ) $ plane defined by the Li\'enard transformation
introduced in \autoref{sec:vdp_intro}. For $ h = 0.0001 $, all methods
show numerical limit cycles resembling the reference solution in
\autoref{fig:vdp_reference}, with solutions ``jumping'' between
branches of the cubic nullcline approximately horizontally at the
critical values $ y _2 = \pm 2/3 $.  For the Euler-type methods,
rather than remaining approximately horizontal as $h$ grows, these
jumps grow in the direction of increasing $ \lvert y _2 \rvert $,
resulting in instability for Euler's method at $ h = 0.01 $ and severe
limit cycle distortion for the exponential Euler and SI Euler
methods. The symplectic Euler method exhibits similar behavior, albeit
less severely. Although the Lie--Trotter splitting method also shows
limit cycle distortion, the jumps oscillate and return to the
nullcline at nearly the correct points. The exponential midpoint and
St\"ormer/Verlet methods both exhibit much less limit cycle
distortion, with $ \lvert y _2 \rvert $ increasing slightly during
jumps for the exponential midpoint method and \emph{decreasing}
slightly for the St\"ormer/Verlet method. The Strang splitting method
also exhibits very little distortion and, like Lie--Trotter, appears
to oscillate and return to the nullcline at nearly the correct point.

\begin{table}
  \centering
  \begin{tabular}{r|rrr|rrr}
    & \multicolumn{3}{c|}{$ \lvert y _1 \rvert $} & \multicolumn{3}{c}{$ \lvert y _2 \rvert $} \\
    $h$ & 0.0001 & 0.001 & 0.01 & 0.0001 & 0.001 & 0.01 \\
    \hline
    Euler & 2.01 & 2.03 & ------ & 0.68 & 0.77 & ------ \\
    Exponential Euler & 2.01 & 2.07 & 3.18 & 0.69 & 0.88 & 7.52 \\
    SI Euler & 2.01 & 2.10 & 4.34 & 0.70 & 0.99 & 22.82 \\
    Exponential Midpoint & 2.00 & 2.00 & 2.07 & 0.68 & 0.68 & 0.87 \\
    Lie--Trotter & 2.00 & 2.00 & 2.00 & 0.68 & 0.68 & 0.68 \\
    Symplectic Euler & 2.01 & 2.03 & 2.37 & 0.68 & 0.77 & 2.06 \\
    Strang & 2.00 & 2.00 & 2.00 & 0.68 & 0.68 & 0.68 \\
    St\"ormer/Verlet & 2.00 & 2.00 & 1.97 & 0.68 & 0.67 & 0.57
  \end{tabular}
  \caption{Values of $ \lvert y _1 \rvert $, $ \lvert y _2 \rvert $ at
    which numerical solutions return to the cubic nullcline after
    jumping, for the Van~der~Pol oscillator with $\epsilon = 50
    $. These values increase with $h$ for the Euler-type methods,
    increase more modestly for the symplectic Euler method, increase
    slightly for the exponential midpoint method, and decrease
    slightly for the St\"ormer/Verlet method. The Lie--Trotter and
    Strang splitting methods exhibit no drift in these values at the
    level of precision shown.\label{tab:vdp_jumps}}
\end{table}

\autoref{tab:vdp_jumps} quantifies the observations made in the
previous paragraph by showing the values of $ \lvert y _1 \rvert $,
$ \lvert y _2 \rvert $ at which the jumps return to the
nullcline. This is computed by finding the point $ ( y _1 , y _2 ) $
along each numerical solution at which $ \lvert y _1 \rvert $ attains
a maximum. For the Euler-type, exponential midpoint, and symplectic
Euler methods, these values increase as $h$ increases, while
St\"ormer/Verlet shows a small decrease. By contrast, the two
splitting methods do not show any drift at the precision displayed.

\begin{figure}
  \centering
  \includegraphics[width=\textwidth]{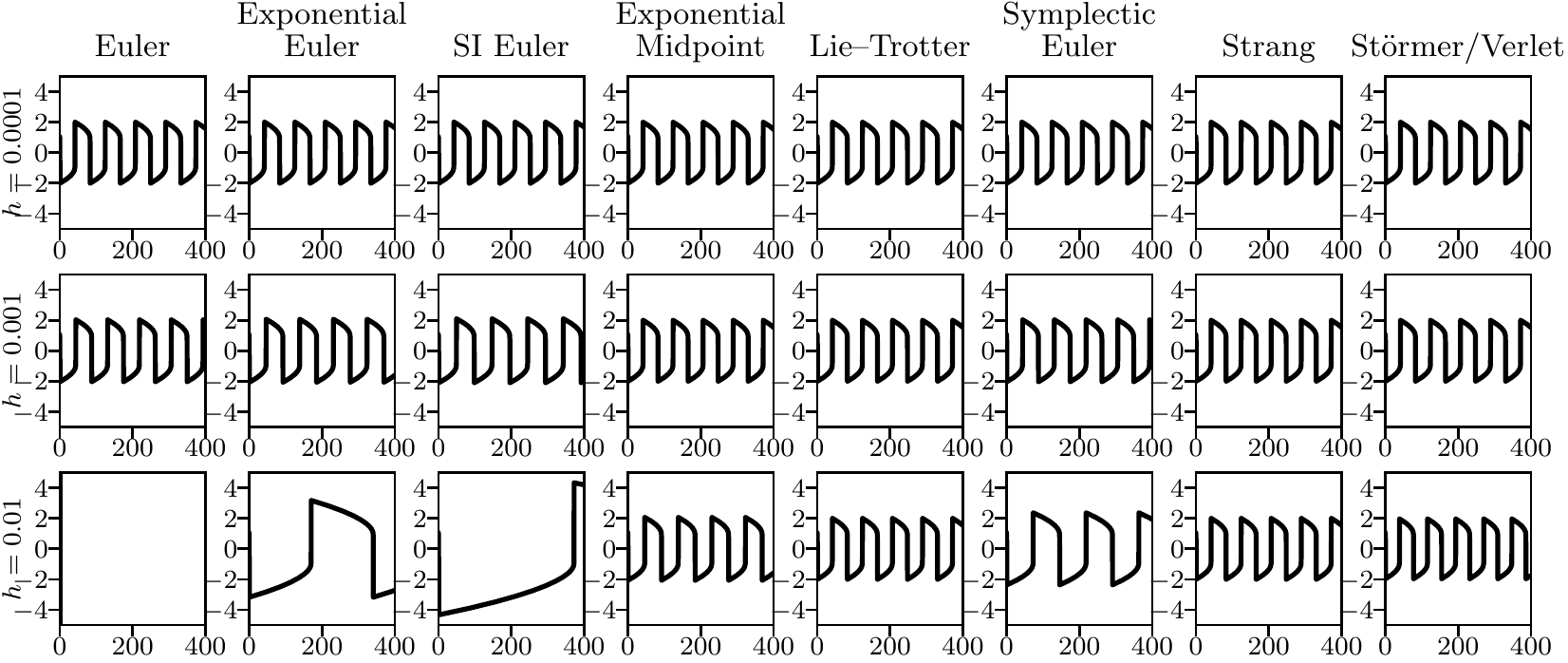}
  \caption{Numerical spiking behavior ($ x _1 $ vs.~$t$) for the
    Van~der~Pol oscillator with $ \epsilon = 50 $. As the time step
    size grows, Euler's method becomes unstable, and the stable
    Euler-type methods exhibit a severe decrease in spike frequency
    and increase in spike amplitude. This frequency and amplitude
    drift is moderate for symplectic Euler and slight for exponential
    midpoint and St\"ormer/Verlet, while the splitting methods show no
    apparent drift in spiking behavior at
    all. \label{fig:vdp_stiff_time}}
\end{figure}

\autoref{fig:vdp_stiff_time} shows time plots for these same numerical
solutions. As $h$ increases, the Euler-type methods show severely
decreased spike frequency and increased spike amplitude. Symplectic
Euler exhibits this same behavior less severely and exponential
midpoint only slightly, while St\"ormer/Verlet exhibits a slight
\emph{increase} in spike frequency. By contrast, the Lie--Trotter and
Strang splitting methods have no apparent change in spike frequency or
amplitude.

This behavior is explained by the preceding observations about the
points at which jumps return to the nullcline. For the stiff
Van~der~Pol oscillator, jumps occur quickly, so nearly all of the time
is spent moving slowly along the nullcline. Numerical solutions that
return to the nullcline at too-large values of $ \lvert y _1 \rvert $,
$ \lvert y _2 \rvert $ must spend more time moving along the nullcline
between jumps, resulting in decreased spike frequency; this is the
case for the Euler-type, symplectic Euler, and exponential midpoint
methods. Likewise, those which return at too-small values of
$ \lvert y _1 \rvert $, $ \lvert y _2 \rvert $ spend less time moving
along the nullcline, resulting in increased spike frequency; this is
the case for St\"ormer/Verlet. This also explains why the Lie--Trotter
method preserves the correct spiking behavior, in spite of the limit
cycle distortion observed in \autoref{fig:vdp_stiff}: this distortion
occurs almost entirely \emph{during} the jumps, where the solution
spends very little time.

\section{Limit cycle preservation for Hodgkin--Huxley neurons}

\subsection{The Hodgkin--Huxley model}
\label{sec:hh_intro}

Based on electrophysiology experiments, \citet{HoHu1952} proposed a
model, consisting of a nonlinear system of partial differential
equations, to describe the dynamics of the membrane potential of the
squid giant axon. If the membrane potential is assumed to be uniform
in space along the axon, the Hodgkin--Huxley model reduces to a
conditionally linear system of ODEs:
\begin{equation}
\label{eqn:hh}
\begin{aligned}
  C \dot{ V } &= I - \bar{ g } _{\mathrm{K}} n ^4 ( V - E
  _{\mathrm{K}} ) - \bar{g}_{\mathrm{Na}} m ^3 h ( V - E
  _{\mathrm{Na}}) - \bar{g} _{\mathrm{L}} ( V - E _{\mathrm{L}}
  ), \\
  \dot{ n } &= \alpha _n (V) ( 1 - n ) - \beta _n ( V ) n ,\\
  \dot{ m } &= \alpha _m (V) ( 1 - m ) - \beta _m ( V ) m ,\\
  \dot{ h } &= \alpha _h (V) ( 1 - h ) - \beta _h ( V ) h .
\end{aligned}
\end{equation}
This describes how $V$, the voltage across a membrane with capacitance
$C$, responds to an input current $I$. (In a neural network, $I$
depends on the membrane voltage of neighboring neurons connected by
synapses. Therefore, a network of Hodgkin--Huxley neurons is also
conditionally linear.) The constants $ \bar{ g } _{\mathrm{K}} $,
$ \bar{ g } _{\mathrm{Na}} $, $ \bar{ g } _{\mathrm{L}} $ and
$ E _{ \mathrm{K} } $, $ E _{ \mathrm{Na} } $, $ E _{ \mathrm{L} } $
are, respectively, the conductances and reversal potentials for the
potassium (K), sodium (Na), and leak (L) channels. The other dynamical
variables, $n$, $m$, $h$, are dimensionless auxiliary quantities
between $0$ and $1$, corresponding to potassium channel activation,
sodium channel activation, and sodium channel inactivation;
$ \alpha _n $, $ \alpha _m $, $ \alpha _h $ and $ \beta _n $,
$ \beta _m $, $ \beta _h $ are given rate functions of $V$.

For the remainder of this section, we take units of mV for $V$ and ms
for $t$ and consider model neurons with the parameters
\begin{equation*}
  \bar{ g } _{\mathrm{K}} = 36, \quad \bar{ g } _{\mathrm{Na}} = 120, \quad \bar{ g } _{\mathrm{L}}  = 0.3, \quad E _{\mathrm{K}} = - 77, \quad E _{ \mathrm{Na} } = 55 , \quad E _{\mathrm{L}} = - 61,
\end{equation*} 
and rate functions
\begin{align*}
  \alpha _n (V) &= \frac{ 0.01 (10 - 65 - V) }{ \exp \bigl( \frac{ 10 - 65 - V}{ 10 } \bigr) - 1 } , & \beta _n (V) &= 0.125 \exp \biggl( \frac{- 65 - V }{ 80 } \biggr) ,\\
  \alpha _m (V) &= \frac{ 0.1 ( 25 - 65 - V ) }{ \exp \bigl( \frac{ 25 - 65 - V }{ 10 } \bigr) - 1 } ,& \beta _m (V) &= 4 \exp \biggl( \frac{ - 65 - V }{ 18 } \biggr),\\
  \alpha _h (V) &= 0.07 \exp \biggl( \frac{ - 65 - V }{ 20 } \biggr) , & \beta _h (V) &= \frac{ 1 }{ \exp \bigl( \frac{ 30 - 65 - V }{ 10 } \bigr) + 1 } .
\end{align*}
These rate functions agree with those in \citet{HoHu1952} with a
resting potential of $ -65 $ mV.

\begin{figure}
  \centering
  \includegraphics[width=\textwidth]{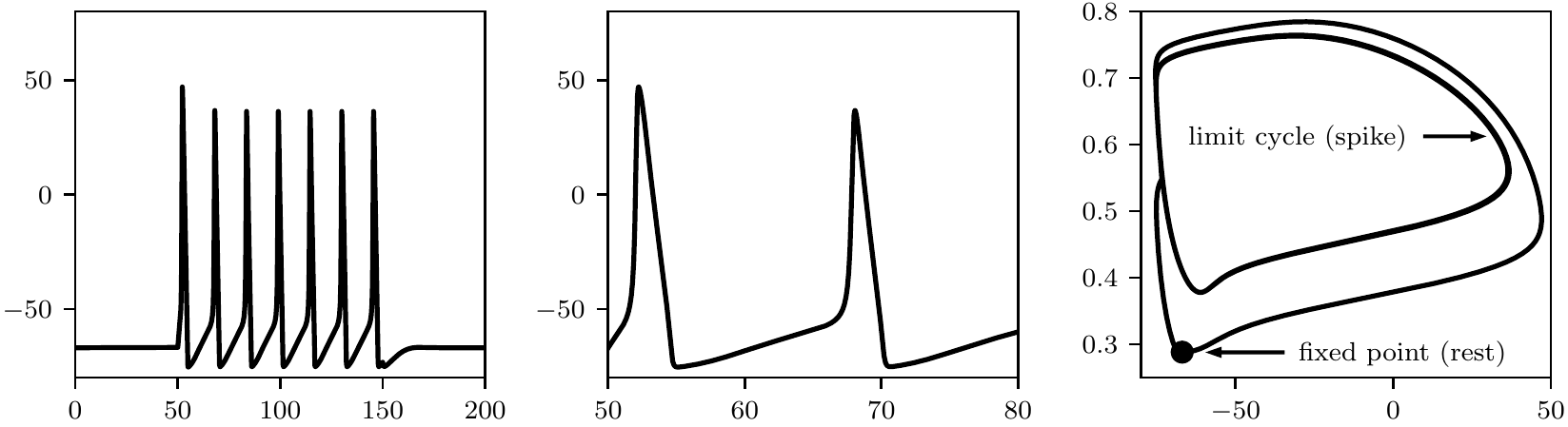}
  \caption{Reference solution for a single Hodgkin--Huxley neuron,
    showing voltage $V$ in mV as a function of time $t$ in ms (left),
    a zoomed-in view of the first two spikes (center), and the
    corresponding curve in the $ ( V, n ) $ phase plane (right). An
    input current of 10~nA is switched on at $ t = 50 $, causing the
    neuron to begin firing periodically, and is switched off again at
    $ t = 150 $, causing it to return to a resting equilibrium. In the
    phase plane, the resting state corresponds to a fixed point, while
    the spiking state corresponds to a limit cycle. The outer
    loop in the phase portrait traces the transition between resting
    and spiking, including the large initial spike when the current is
    first switched on.
    \label{fig:hh_reference}}
\end{figure}

\autoref{fig:hh_reference} shows a reference solution to
\eqref{eqn:hh}, illustrating the bifurcation and limit cycle behavior
that governs neuronal spiking, as discussed in
\autoref{sec:intro}. Since the limit cycle lies on a two-dimensional
center manifold (\citet{Hassard1978,Izhikevich2007}), these dynamics
can be portrayed by projecting to a two-dimensional phase plot in the
$ ( V, n ) $ plane.

The limit cycle, together with a time parametrization, determines the
shape of action potentials. Since we are interested in preserving
these features, it is worth discussing the question: \emph{Is the
  shape of an action potential an important feature for a neuron model
  to capture?} It has been well established that the shape of an
action potential can and does vary between neurons. For example,
regular-spiking pyramidal neuron action potentials are broader and
differ detectably from the narrow, fast-spiking interneurons
(\citet{Bean2007,NoAzSaGrMc2003}). These differences are commonly
exploited during extracellular recordings for identifying the neuron
type; further, even for the same neuron, the shape of the action
potential changes depending on whether the recording is made from
axonal or somatic compartments \citep{Bean2007}. Therefore, faithfully
capturing this neural response feature would be important for modeling
the diversity of cell types typically found in biological neural
networks, and even for developing tightly constrained
multi-compartment models of single neurons. For developing models of
neural networks, the effect of the action potential on the
post-synaptic neuron must be considered. The effect of an action
potential from a pre-synaptic onto the post-synaptic neuron can differ
depending on the type of synapses (chemical vs.\ electrical) between
them (\citet{DeMaSe1994,CuOB2016}). Would differences in action
potential shapes, at the level of individual neurons, alter
information processing and other emergent properties, such as neural
synchronization, at a network level? This is not fully understood and
is a direction for future investigation.

\subsection{Application of numerical methods}

Since \eqref{eqn:hh} is conditionally linear, the application of
Euler-type methods is straightforward; so is the application of other
explicit exponential Runge--Kutta methods, such as the exponential
midpoint method. The use of the exponential Euler method for
Hodgkin--Huxley was proposed by \citet{MoRa1974}, and it has since
been in widespread use in computational neuroscience, including as the
default numerical integrator in the GENESIS software package
(\citet{BoBe1998}). \citet{BuMc2004} showed that exponential Euler
often gives larger errors than Euler's method for moderately large
time steps where both methods are stable. \citet{BoNe2013}, in
addition to exponential Euler, also considered the {SI} Euler and
exponential midpoint methods for Hodgkin--Huxley.

In order to apply the splitting and composition methods of
\autoref{sec:splitting} to Hodgkin--Huxley, we begin by splitting the
vector field $f$ defining the right-hand side of \eqref{eqn:hh} into
\begin{equation*}
  f = f ^{ (V) } + f ^{ (n) } + f ^{ (m) } + f ^{ (h) }.
\end{equation*}
However, the vector fields $ f ^{ ( n ) } $, $ f ^{ (m) } $, and
$ f ^{ (h) } $ commute, since when $V$ is held fixed, the variables
$n$, $m$, $h$ evolve independently. Taking
$ f ^{ ( n, m, h) } = f ^{ (n) } + f ^{ (m) } + f ^{ (h) } $, we have
$ \Phi _h ^{ (n,m,h) } = \Phi _h ^{ (n) } \circ \Phi _h ^{ (m) } \circ
\Phi _h ^{ (h) } $ for any exact or approximate flow $ \Phi _h $, and
we can evolve these flows in parallel (see
\autoref{rmk:commuting_flows}).

Therefore, in spite of the fact that \eqref{eqn:hh} is
four-dimensional, we can construct splitting and composition methods
using only two flows. As before, we consider non-symmetric and
symmetric splitting/composition methods:
\begin{equation*}
  \Phi _h ^{ (V) } \circ \Phi _h ^{ (n,m,h) } , \qquad \Phi _{ h/2 } ^{ (n,m,h) \ast } \circ \Phi _{ h /2 } ^{(V)\ast} \circ \Phi _{ h/2 } ^{ (V) } \circ \Phi _{h/2} ^{ (n,m,h) }  .
\end{equation*}
Taking $ \Phi _h ^{ ( V ) } = \varphi _h ^{ ( V) } $ and
$ \Phi _h ^{ ( n, m , h ) } = \varphi _h ^{ (n,m,h) } $ to be the
exact flows, the non-symmetric method is again the Lie--Trotter
splitting method, and the symmetric method is the Strang splitting
method. On the other hand, if we take $ \Phi _h ^{ (V) } $ to be
Euler's method and $ \Phi _h ^{ (n,m,h) } $ to be the backward Euler
method, then we again refer to the resulting composition methods as
symplectic Euler and St\"ormer/Verlet, respectively. We refer back to
\autoref{ex:composition} and \autoref{ex:splitting} for the explicit
formulas for these methods.

As with the $ {d} = 2 $ symmetric splitting/composition methods, we
may reuse the nonlinear function evaluation from
$ \Phi _{ h/2 } ^{ (V) } $ for $ \Phi _{ h / 2 } ^{ (V) \ast } $, and
likewise from $ \Phi _{ h / 2 } ^{ (n,m,h) \ast } $ for
$ \Phi _{ h / 2 } ^{ (n,m,h) } $. Therefore, with the exception of the
very first half-step of the symmetric method, both the non-symmetric
and symmetric splitting/composition methods require only one
evaluation of each nonlinear function per step, just as for the
Euler-type methods.

\begin{remark}
  \label{rmk:hines}
  The St\"ormer/Verlet method can be seen as alternating between the
  time-$h$ flows
  $ \Phi _{ h / 2 } ^{ (V) \ast } \circ \Phi _{ h / 2 } ^{ (V) } $ and
  $ \Phi _{ h / 2 } ^{ (n,m,h) } \circ \Phi _{ h / 2 } ^{ (n,m,h) \ast
  } $, both of which correspond to the trapezoid method. If we ignore
  the half-steps, this can be seen as a staggered-grid method with $V$
  stored at integer time steps and $ ( n , m , h ) $ at half-integer
  time steps. This perspective, with $V$ and $ (n,m,h) $ stored at
  staggered time steps and advanced using the trapezoid method, was
  suggested by \citet[Eq.~8]{Hines1984} and later appeared in
  \citet[p.~599--600]{MaSh1998}.

  The St\"ormer/Verlet formulation above has the advantage of
  producing values of $ ( V, n, m , h ) $ all at integer time steps,
  rather than only at staggered time steps. A slightly different
  non-staggered formulation of Hines' method was
  recently proposed by \citet{Hanke2017}, using Euler's method instead
  of backward Euler for $ \Phi _{h/2} ^{ (n,m,h) } $.
\end{remark}

\subsection{Numerical experiments}
\label{sec:hh_numerics}

\begin{figure}
  \centering
  \includegraphics[width=\textwidth]{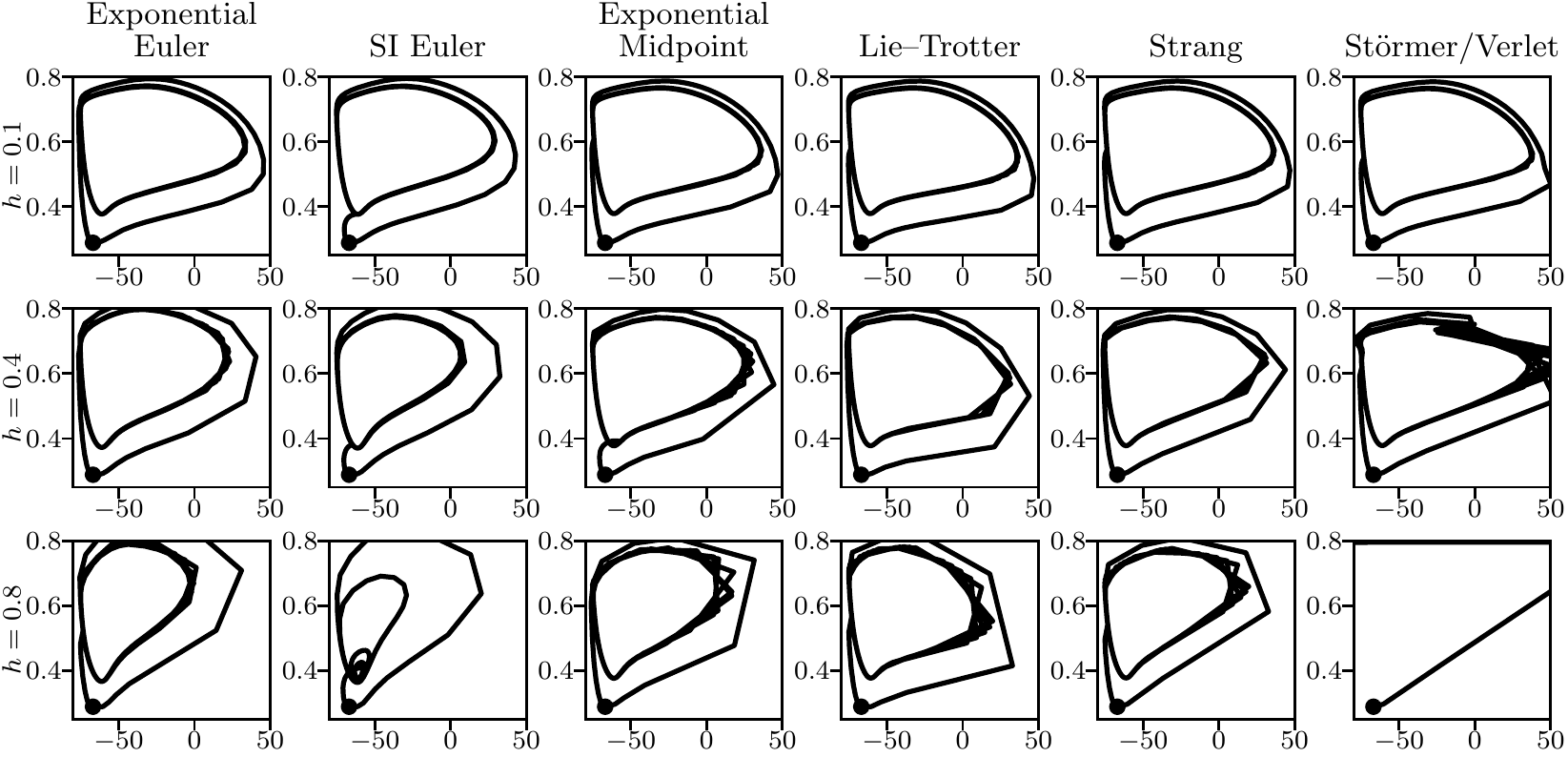}
  \caption{Numerical bifurcation and limit-cycle behavior for the
    Hodgkin--Huxley neuron in the $ ( V, n ) $ phase plane. The
    exponential midpoint and splitting methods show better limit-cycle
    preservation (particularly with respect to shrinking in the $V$
    direction) than the exponential Euler or (especially) SI Euler
    methods as $h$ increases, while the St\"ormer/Verlet method
    becomes numerically unstable.\label{fig:hh_phase}}
\end{figure}

\begin{figure}
  \centering
  \includegraphics[width=\textwidth]{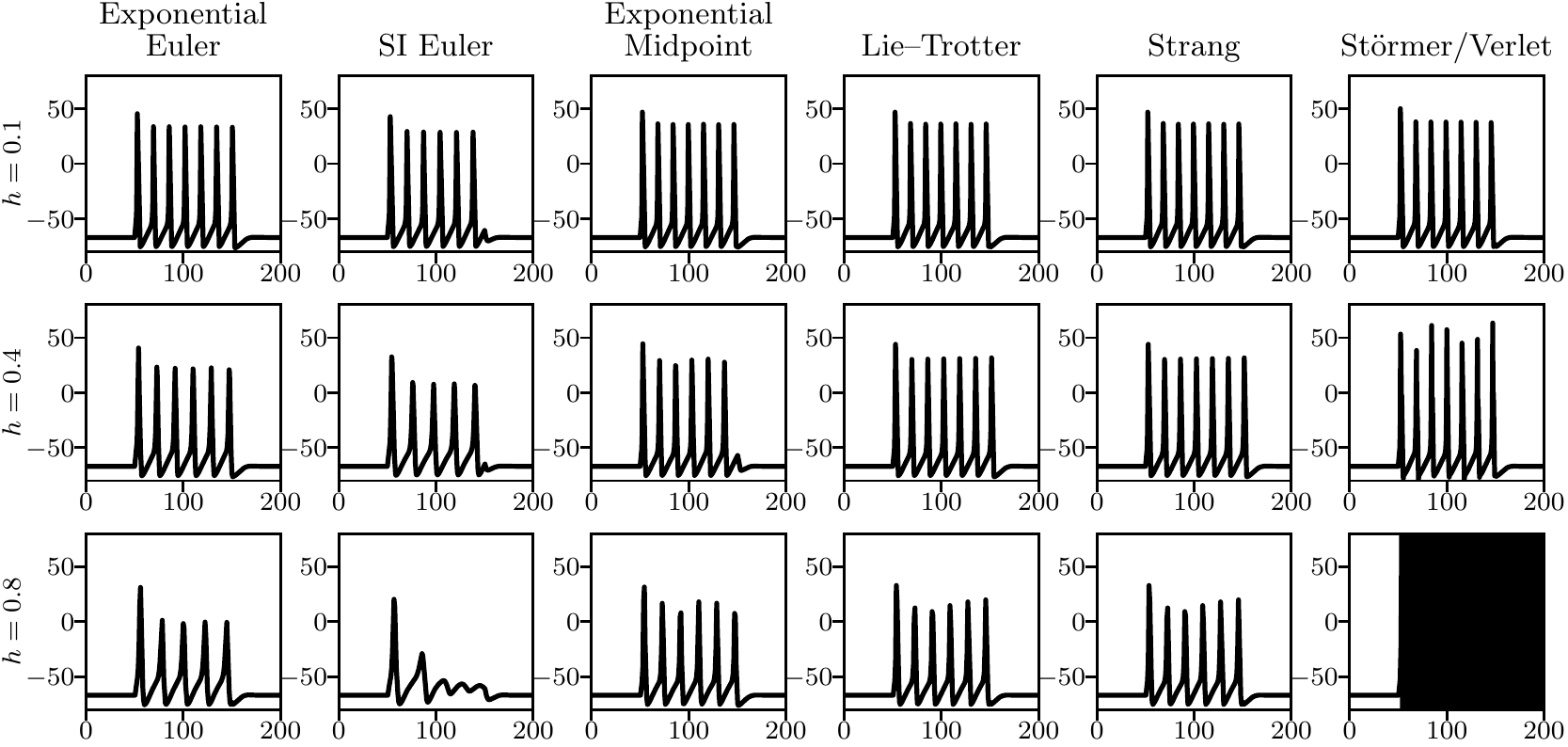}
  \caption{Numerical spiking behavior ($V$ in mV vs.~$t$ in ms) for the
    Hodgkin--Huxley neuron, which is stimulated by a 10~nA input current
    between $ t = 50 $ and $ t = 150 $~ms. The Lie--Trotter and Strang
    splitting methods show less severe decrease in spiking frequency and
    amplitude than either the exponential Euler or (especially) SI Euler
    methods, and very slightly less than the exponential midpoint method. For example, both splitting methods emit the correct number
    of spikes, 7, at $ h = 0 . 4 $, while exponential midpoint and
    exponential Euler emit 6 spikes and SI Euler only 5 spikes. The
    St\"ormer/Verlet method becomes numerically unstable.
    \label{fig:hh_time}}
\end{figure}

\begin{figure}
  \centering
  \includegraphics[width=\textwidth]{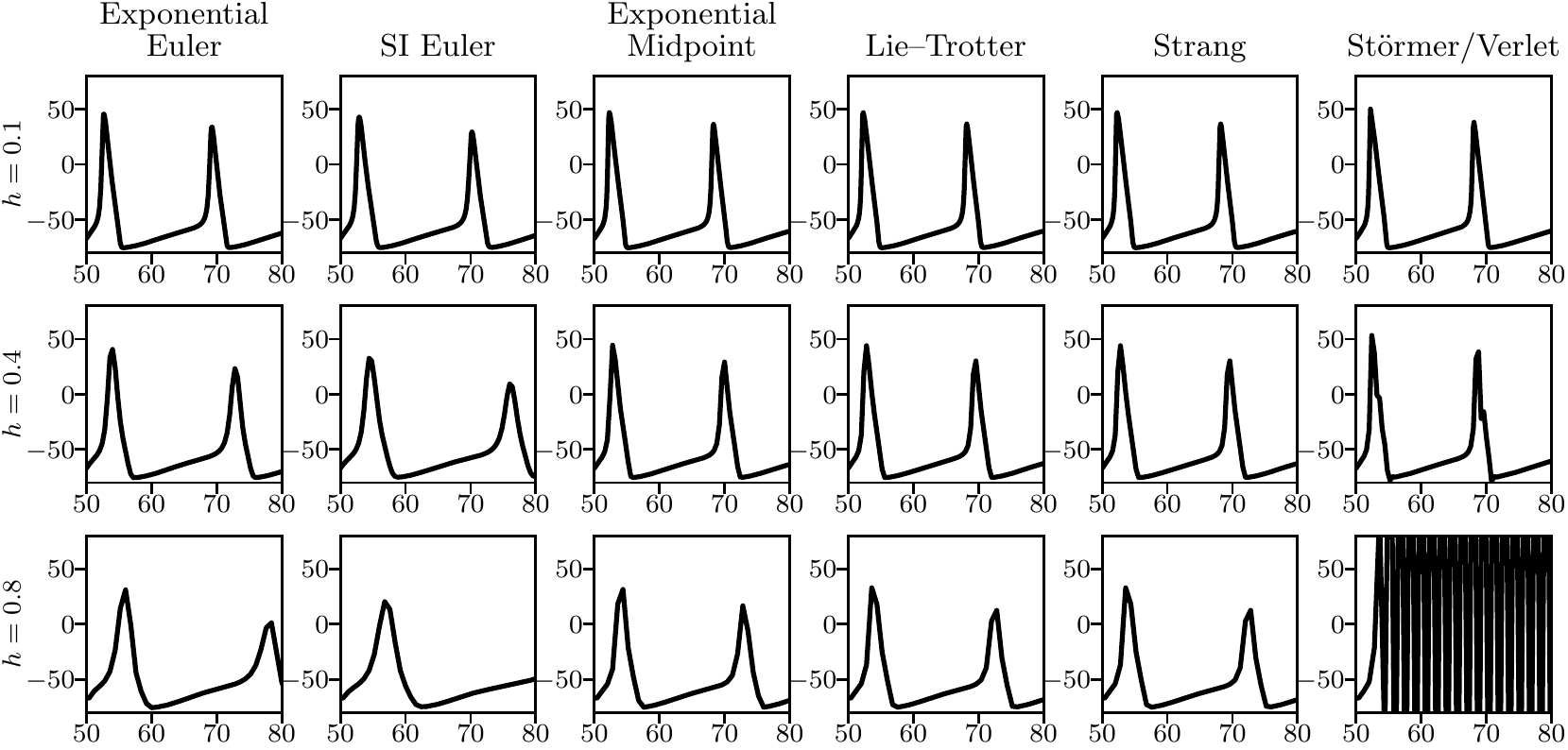}
  \caption{A zoomed-in view of the plots in \autoref{fig:hh_time},
    showing the shape of the first two action potentials. For the
    exponential Euler method, the spikes become less sharp and more
    spread out as $h$ increases, and this effect is even worse for the
    SI Euler method. The exponential midpoint and splitting methods do
    a better job at preserving the shape and width of the spikes,
    while the St\"ormer/Verlet method becomes unstable for large
    $h$.\label{fig:hh_spike}}
\end{figure}

\autoref{fig:hh_phase}, \autoref{fig:hh_time}, and
\autoref{fig:hh_spike} show the results of applying Euler-type
methods, the exponential midpoint method, and splitting/composition
methods to the Hodgkin--Huxley system \eqref{eqn:hh}, with the
parameters and rate functions specified in \autoref{sec:hh_intro}, for
a range of large time step sizes. Euler's method and the symplectic
Euler method both exhibited numerical instability at these time step
sizes, so they do not appear in the figures. The specific problem is
the same one for which the reference solution appears in
\autoref{fig:hh_reference}: the neuron is simulated for 200~ms, with a
constant 10~nA input current switched on at $ t = 50 $~ms and switched
off again at $ t = 150 $~ms. The neuron is initially at rest, begins
spiking after the input current is switched on, and returns to rest
after the input current is switched off.

As they did for the stiff Van~der~Pol oscillator in
\autoref{sec:vdp_stiff}, the exponential Euler and {SI} Euler methods
exhibit limit cycle distortion and decreased spiking frequency as $h$
increases, and this is especially severe for the SI Euler method. For
instance, notice that in the reference solution, the neuron fires 7
spikes in response to the input current. In \autoref{fig:hh_time} the
exponential Euler method fires 7 spikes at $ h = 0.1 $, 6 spikes at
$ h = 0.4 $, and 5 spikes at $ h = 0.8 $, while the SI Euler method
fires only 6 spikes at $ h = 0.1 $ and 5 spikes at $ h = 0.4 $, and
its spiking behavior is essentially damped away at $ h = 0.8 $.
Decreased spiking amplitude is also apparent, which corresponds to the
limit cycle being compressed horizontally, as observed in
\autoref{fig:hh_phase}. Zooming in on the spikes, in
\autoref{fig:hh_spike}, we see that they become less sharp and more
stretched out for these methods as $h$ increases.

By contrast, the Lie--Trotter and Strang splitting methods both appear
to do a better job at preserving the limit cycle, as well as spiking
frequency and amplitude, as $h$ increases. Both splitting methods fire
7 spikes at $ h = 0 .1 $ and $ h = 0. 4 $, and this decreases to 6
spikes at $ h = 0.8 $. Some decrease in spiking amplitude is apparent
at $ h = 0 . 8 $, although at each step size there is less amplitude
decay than for either the exponential Euler or SI Euler
methods. Therefore, although some decrease in spiking frequency and
amplitude occurs for the splitting methods, it appears to be notably
less severe than for these other methods. Likewise, the spikes remain
sharper and less stretched out for the splitting methods than for the
Euler-type methods, although some smearing out of the spike shape is
clearly visible at $ h = 0.8 $.

The exponential midpoint method performs slightly worse than the
splitting methods at the same time step size: it only fires 6 spikes
at $ h = 0.4 $, although it comes close to firing a 7th, and the spike
amplitude exhibits fluctuations. At $ h = 0.8 $, its decrease in
spiking amplitude and frequency is comparable to that of the splitting
methods. Preservation of spike shape is also similar to the splitting
methods. However, since the exponential midpoint method requires
twice as many function evaluations per step, one could use a splitting
method with $ h = 0.4 $ at the same cost as exponential midpoint with
$ h = 0.8 $, and by that comparison the splitting methods are clearly
superior.

Finally, the St\"ormer/Verlet method---which, as we noted in
\autoref{rmk:hines}, can be seen as a non-staggered version of Hines'
method \citep{Hines1984}---becomes unstable as $h$ increases. Although
it exhibits the correct behavior (both in terms of spiking frequency
and amplitude) at $ h = 0 .1 $, the beginnings of instability are
visible at $ h = 0 . 4 $ in the increase and fluctuation in spiking
amplitude, while the method has become unstable by $ h = 0.8 $.

These experiments suggest that the Lie--Trotter and Strang splitting
methods can preserve the correct spiking behavior in Hodgkin--Huxley
neurons for larger time steps than can the exponential Euler or {SI}
Euler methods---and the SI Euler method performs especially poorly in
this regard. The exponential midpoint method performs similarly to the
splitting methods but at twice the cost per time step; after
normalizing for computational cost, the splitting methods come out on
top. Methods involving Euler steps, including Euler's method,
symplectic Euler, and St\"ormer/Verlet, begin to suffer from numerical
instability at these large time step sizes.

\subsection{Remarks on reduced Hodgkin--Huxley models}
The numerical experiments in \citet{BoNe2013} do not actually simulate
the full, four-dimensional Hodgkin--Huxley system
\eqref{eqn:hh}. Rather, they make the simplifying assumption that
sodium activation is ``instantaneous,'' so that
$ m = m _\infty (V) \coloneqq \alpha _m (V) / \bigl( \alpha _m (V) +
\beta _m (V) \bigr) $, which eliminates the differential equation for
$m$ and yields a three-dimensional reduced model. The reduced model is
no longer conditionally linear, since $ m _\infty (V) $ appears in the
differential equation for $V$, but \citet{BoNe2013} deal with this by
freezing $m$ at the beginning of each step, including the intermediate
step of exponential midpoint.

Because of the differences between the full Hodgkin--Huxley model and
the reduced model, there are some discrepancies between the numerical
results presented here and those in \citet{BoNe2013}, especially with
regard to the decreasing amplitude of spikes at large time steps,
which we observe and they do not. The purpose of this section is to
briefly discuss these differences and to point out some of the
dynamical changes that are introduced by using the reduced model.

The primary purpose of reducing the Hodgkin--Huxley model to two or
three dimensions, by assuming instantaneous sodium activation and/or
replacing $n$ and $h$ by a single variable, is to obtain a simpler
system whose qualitative dynamics can be more easily
analyzed. \citet{Meunier1992} compares the dynamics of these reduced
models to those of Hodgkin--Huxley, observing several ``spurious
effects not displayed by the original models'' that ``arise mainly
from the assumption that the variable $m$ follows instantaneously the
variations of $V$,'' and concludes that ``taking into account properly
the dynamics of sodium activation has more importance than choosing a
reduction scheme or another for the slow variables.''  While reduction
is useful for understanding qualitative dynamics, he points out that
``[i]f saving computation time was one of the reasons originally it is
no longer of dominant importance'' and ``it is not meant to be used as
a substitute to the original equations in network simulations.''

\begin{figure}
  \centering
  \includegraphics[width=\textwidth]{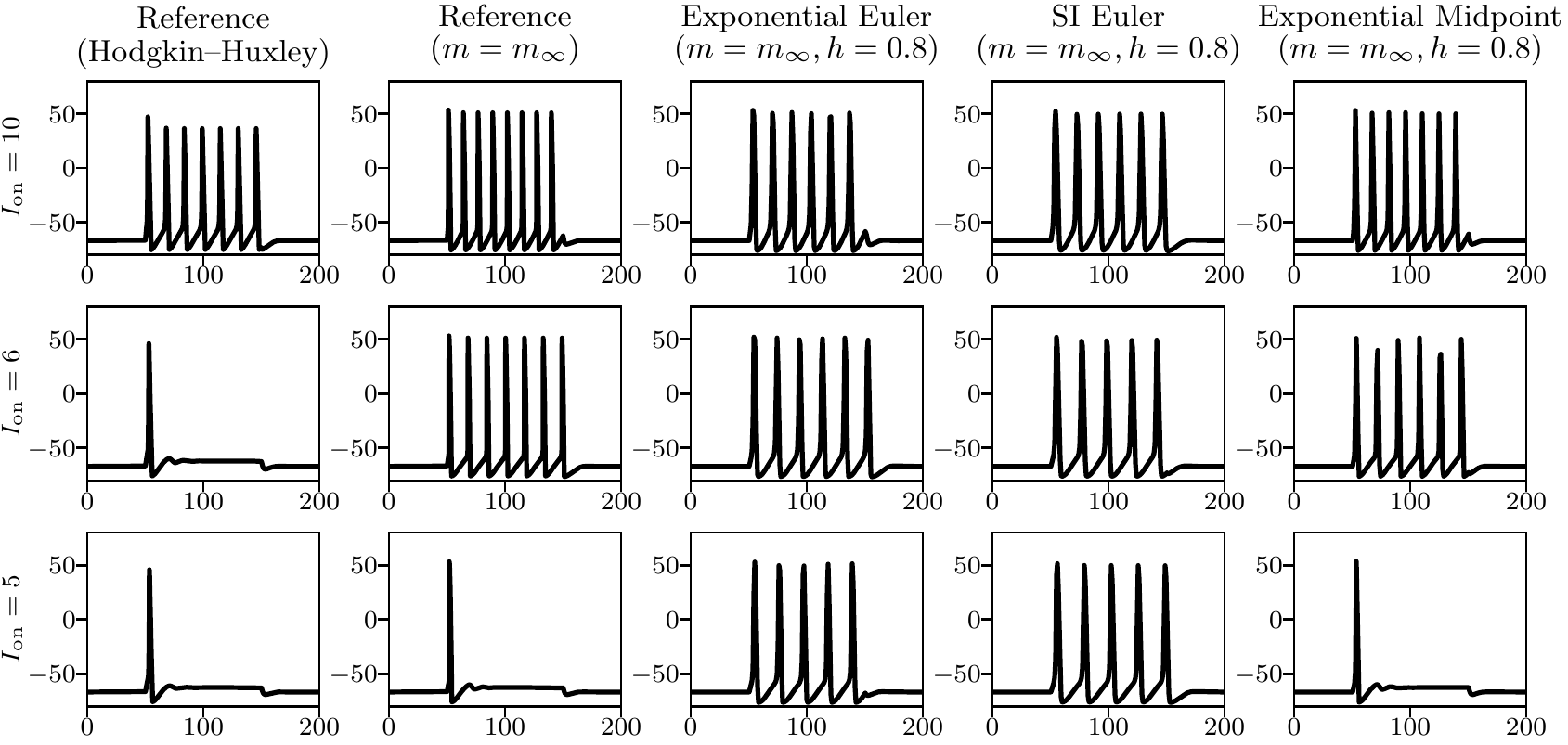}
  \caption{Treating sodium activation as instantaneous
    ($ m = m _\infty $) results in different spiking behavior than in
    the original Hodgkin--Huxley model. For $ I_\text{on} = 10 $, the
    Hodgkin--Huxley model emits 7 spikes, the first one larger than
    the other 6, whereas the $ m = m _\infty $ model emits 8 large
    spikes of roughly the same amplitude. For $ I_\text{on} = 6 $, the
    Hodgkin--Huxley model emits only one spike and returns to rest,
    whereas the $ m = m _\infty $ model emits 7 large spikes. For
    $ I_\text{on} = 5 $, both models emit only one spike. At
    $ h = 0 .8 $~ms, the Euler-type methods and exponential midpoint
    method exhibit decrease in spiking frequency---but not in
    amplitude, as they did with the full Hodgkin--Huxley model---and
    the Euler-type methods spike spuriously at
    $ I_\text{on} = 5 $.\label{fig:hh_minfty}}
\end{figure}

\autoref{fig:hh_minfty} illustrates some of the ``spurious effects''
that can result from treating sodium activation as instantaneous,
including spiking at subthreshold input currents.  We perform the same
numerical experiment as in \autoref{sec:hh_numerics}, switching on an
input current $ I_\text{on} $ between $ t = 50$ and $ t = 150 $~ms. In
addition to $ I _\text{on} = 10 $~nA, which was used in the previous
experiments, we also consider currents of 6~nA and 5~nA, which are
subthreshold for the original Hodgkin--Huxley model. Only the
exponential Euler, SI Euler, and exponential midpoint methods of
\citet{BoNe2013} are shown, since the splitting/composition methods
fundamentally require conditional linearity.

Finally, we remark that, in spite of its dynamical differences from
Hodgkin--Huxley, this reduced model is widely used and useful. The
exponential Euler, SI Euler, and exponential midpoint methods are well
suited to it, whereas the splitting/composition methods are only
applicable to the full Hodgkin--Huxley model. Whether or not to use
instantaneous sodium activation is ultimately a modeling question
rather than a numerical one, and it should not be used numerically if
one actually wishes to simulate the full Hodgkin--Huxley dynamics.

\section{Conclusion}

For conditionally linear systems, such as the Van~der~Pol oscillator
and Hodgkin--Huxley model of neuronal dynamics, the exponential Euler
and SI Euler methods remain stable at much larger time step sizes than
the classical Euler method, with no additional computational cost per
step. However, we have shown that these Euler-type methods preserve
the limit cycles of these systems poorly as the time step size grows,
producing oscillations with the wrong amplitude and/or
frequency. (This adds to previous observations about the inaccuracy of
exponential Euler with large time steps for the Hodgkin--Huxley model,
cf.~\citet{BuMc2004}.) Since limit cycles underlie the dynamics of
neuronal spiking, this is a serious problem for the use of Euler-type
methods with large time steps.

By contrast, we have shown that splitting and composition methods
exhibit better limit cycle preservation for these systems as time step
size grows, with no additional cost per step. For the non-stiff
Van~der~Pol oscillator, the (first-order, non-symmetric) Lie--Trotter
splitting and symplectic Euler composition methods perform similarly,
as do the (second-order, symmetric) Strang splitting and
St\"ormer/Verlet methods. However, for the stiff Van~der~Pol
oscillator and Hodgkin--Huxley models, the splitting methods both do a
better job of preserving the frequency and amplitude of oscillations,
and maintaining numerical stability, than the corresponding
composition methods. With respect to these properties, the splitting
methods are also seen to outperform the exponential midpoint method,
which requires twice as many function evaluations per step.

Of all the methods considered, the Strang splitting method exhibits
the best performance across the spectrum of non-stiff and stiff
systems. This method also has the advantage of being second-order,
although its order is \emph{not} the main reason for its superior
performance; this is evinced by the fact that the first-order
Lie--Trotter splitting method performs nearly as well and outperforms
the Euler-type methods having the same order. The Strang splitting
method ought to be seriously considered as a competitor to exponential
Euler as the ``standard'' integration method for such problems.

Although we have focused on the performance of these integrators for
individual Hodgkin--Huxley neurons, future work should look at the
implementation and performance of the Strang splitting method for
networks of neurons. For large networks, \autoref{sec:euler-splitting}
lays out a hybrid Euler-type/splitting approach by which such a system
could be partitioned for efficient parallel implementation.  Another
direction for future work involves exponential integrators based on
local linearization, such as the exponential Rosenbrock--Euler method
(cf.~\citet{HoOs2010}), rather than coordinate splitting.

\subsection*{Acknowledgments}
Zhengdao Chen was supported in part by an ARTU research fellowship in
the Department of Mathematics and Statistics at Washington University
in St.~Louis. Baranidharan Raman was supported in part through an NSF
CAREER grant (\#1453022). Ari Stern was supported in part by grants
from the NSF (DMS-1913272) and from the Simons Foundation
(\#279968). We also wish to thank the anonymous referees for their
helpful comments and suggestions.

\end{document}